\documentclass[UTF-8,reqno,12pt]{amsart}
\usepackage{enumerate}
\usepackage{mhequ}

\usepackage[x11names]{xcolor}

\usepackage{bbm}
\usepackage{geometry}
\usepackage[normalem]{ulem}
\usepackage{dutchcal}

\usepackage{enumitem}  

\usepackage{amssymb,url,color, booktabs,nccmath}

\usepackage{mathrsfs}
\usepackage{graphicx}
\usepackage{tikz}
\usepackage{amsthm}
\usepackage{comment}
\usepackage{amsmath}

\usepackage{relsize}

\usepackage[T1]{fontenc}

\usepackage{color}
\usepackage[colorlinks=true]{hyperref}
\hypersetup{
linkcolor=blue,          
citecolor=red,        
filecolor=blue,      
urlcolor=cyan
}

\makeatletter
\renewcommand{\subsection}{\@startsection
{subsection} 
{2} 
{0mm} 
{0.5\baselineskip} 
{0.3\baselineskip} 
{\normalfont\normalsize\raggedright}} 
\makeatother

\allowdisplaybreaks[4]

\usepackage{fancyhdr} 
\numberwithin{equation}{section}

\newcommand{\be}{\begin{eqnarray}}
\newcommand{\ee}{\end{eqnarray}}
\newcommand{\ce}{\begin{eqnarray*}}
\newcommand{\de}{\end{eqnarray*}}
\newtheorem{theorem}{Theorem}[section]
\newtheorem{lemma}[theorem]{Lemma}
\newtheorem{remark}[theorem]{Remark}
\newtheorem{definition}[theorem]{Definition}
\newtheorem{proposition}[theorem]{Proposition}
\newtheorem{example}[theorem]{Example}
\newtheorem{corollary}[theorem]{Corollary}

\newtheorem*{theorem*}{Theorem}
\newtheorem*{remark*}{Remark}

\def\geq{\geqslant}
\def\leq{\leqslant}
\def\ge{\geqslant}
\def\le{\leqslant}

\def\div{\mathord{{\rm div}}}

\def\bH{{\mathbf H}}

 \def\R{\mathbb R}
\def\ff{\frac} \def\R{\mathbb R}  \def\ff{\frac}  
\def\N{\mathbb N}  
   
\def\<{\langle} \def\>{\rangle}

\allowdisplaybreaks

\def\eps{\varepsilon}

\def\e{\mathrm{e}}

\def\p{\partial}
\def\d{\hh}

\def\[{{\Big[}}
\def\]{{\Big]}}
\def\<{{\langle}}
\def\>{{\rangle}}

\def\dif{{\mathord{{\rm d}}}}

\def\bb2{{\boldsymbol{2}}}
\def\no{\nonumber}
\def\={&\!\!=\!\!&}

\def\bB{{\mathbf B}}

\def\bC{{\mathbf C}}

\def\cM{{\mathcal M}}
\def\cN{{\mathcal N}}

\def\cP{{\mathcal P}}

\def\cR{{\mathcal R}}

\def\cU{{\mathcal U}}

\def\mE{{\mathbb E}}

\def\mI{{\mathbb I}}

\def\mN{{\mathbb N}}

\def\mP{{\mathbb P}}

\def\mR{{\mathbb R}}

\def\b1{{\mathbbm 1}}

\def\sL{{\mathscr L}}

\def\sS{{\mathscr S}}

\def\E{\mathbb E}

\def\geq{\geqslant}
\def\leq{\leqslant}
\def\ge{\geqslant}
\def\le{\leqslant}

\def\div{\mathord{{\rm div}}}

\def\eps{\varepsilon}

\def\e{\mathrm{e}}

\def\p{\partial}
\def\d{\hh}

\def\[{{\Big[}}
\def\]{{\Big]}}
\def\<{{\langle}}
\def\>{{\rangle}}

\def\dif{{\mathord{{\rm d}}}}

\def\no{\nonumber}
\def\={&\!\!=\!\!&}
\def\bt{\begin{theorem}}
\def\et{\end{theorem}}
\def\bl{\begin{lemma}}
\def\el{\end{lemma}}
\def\br{\begin{remark}}
\def\er{\end{remark}}
\def\bd{\begin{definition}}
\def\ed{\end{definition}}
\def\bp{\begin{proposition}}
\def\ep{\end{proposition}}
\def\bc{\begin{corollary}}
\def\ec{\end{corollary}}

\def\geq{\geqslant}
\def\leq{\leqslant}
\def\ge{\geqslant}
\def\le{\leqslant}

\def\div{\mathord{{\rm div}}}

\def\bH{{\mathbf H}}

 \def\R{\mathbb R}
\def\ff{\frac} \def\R{\mathbb R}  \def\ff{\frac}  
\def\N{\mathbb N}  
 
\def\<{\langle} \def\>{\rangle}

\def\bb2{{\boldsymbol{2}}}
\def\bbb1{\boldsymbol{1}}

\def\no{\nonumber}
\def\={&\!\!=\!\!&}

\def\1{\mathbbm{1}}
\def\e{\mathrm{e}}
\def\levy{L\'{e}vy }
\def\ito{It$\hat{\mathrm{o}}$}

\def\P{\mathbb{P}}
\def\d{\mathrm{d}}

\def\nalpha{{\alpha}} 
\def\ind{{~\hbox{\rm l\kern-.4em\hbox{\rm l}}}~} 

\def\E{\mathbb{E}}
\def\P{\mathbb{P}}
\def\R{\mathbb{R}}

\def\d{\mathrm{d}}
\def\N{\mathbb{N}}

\def\1{\mathbbm{1}}
\def\e{\mathrm{e}}
\def\levy{L\'{e}vy}
\def\ito{It$\hat{\mathrm{o}}$}

\begin{document}

\title[Quantitative approximation related to $\alpha$-stable noise]{Quantitative approximation to density dependent SDEs driven by $\alpha$-stable processes}

\author{Ke Song}
\author{Zimo Hao}
\author{Mingkun Ye}

\address{Department of Mathematics, Beijing Institute of Technology, Beijing 100081, China}
\email{songke@bit.edu.cn}

\thanks{Ke Song is grateful to the financial supports by National Key R \& D Program of China (No. 2022YFA1006300) and the financial supports of the NSFC (No. 12426205, No. 12271030).}

\address{Universit\"at Bielefeld, Fakult\"at f\"ur Mathematik, Bielefeld 33615, Germany }
\email{zhao@math.uni-bielefeld.de}

\thanks{Zimo Hao is grateful for the DFG through the CRC 1283/2 2021 - 317210226 "Taming uncertainty and profiting from randomness and low regularity in analysis, stochastics and their applications".}

\address{School of Mathematics, Sun Yat-sen University, Guangzhou 510275, China}
\email{yemk@mail2.sysu.edu.cn}


\begin{abstract}
{
Based on a class of moderately interacting particle systems, we establish a quantitative approximation for density-dependent McKean-Vlasov SDEs and the corresponding nonlinear, nonlocal PDEs. The SDE is driven by both Brownian motion and pure-jump L\'{e}vy processes. By employing Duhamel's formula, density estimates, and appropriate martingale functional inequalities, we derive precise convergence rates for the empirical measure of particle systems toward the law of the McKean–Vlasov SDE solution. Additionally, we quantify both weak and pathwise convergence between the one-marginal particle and the solution to the McKean-Vlasov SDE. Notably, all convergence rates remain independent of the noise type.
}
	\\ 
\end{abstract}


\date{\today}

\maketitle

\setcounter{tocdepth}{2}
\tableofcontents

\section{Introduction}
Following the seminal work of McKean \cite{Mc67} and Kac \cite{kac1956foundations}, there has been a growing interest in investigating the McKean-Vlasov stochastic differential equations (SDEs), also known as distribution-dependent SDEs (DDSDEs), or mean-field SDEs, and their corresponding non-linear partial differential equations (PDEs).

In this paper, we consider the following $d$-dimensional density dependent SDE (dDSDE) driven by $\alpha$-stable processes with $\alpha\in(1,2]$:
\begin{align}\label{EQ:dDSDE:01}
	\dif X_t &=b(t,X_t,\rho_t(X_t)) \dif t + \dif L^{\nalpha}_t,\quad t\in[0,T],
\end{align}
where $b:\mR_{+}\times \mR^d\times \mR_+\to \mR^d$ is measurable, $\rho_t(x):=\mu_{X_t}(\mathrm{d} x)/{\mathrm{d} x} (x)$ is the density of the time marginal law $\mu_{X_t}$ to the solution \((X_t)_{t\ge0}\)  with respect to the Lebesgue measure, and $(L^{\nalpha}_t)_{t\ge0}$ is a standard $\R^d$-valued $\alpha$-stable process defined on some probability space $(\Omega, \mathcal{F}, \mathbb{P})$. When $\alpha=2$, $(L^{2}_t)_{t\ge0}$ denotes the standard $d$-dimensional Brownian motion.

By applying It\^o's formula, the density \(\rho_t\) solves the following nonlinear PDE:  
\begin{align}\label{FPE0}
    \p_t \rho_t=\Delta^{\frac{\alpha}{2}}\rho_t-\div (b(\rho_t)\rho_t),
\end{align}  
where for $\alpha=2$, $\Delta $ is the standard  Laplacian operator, and for $\alpha\in(1,2)$, $\Delta^{\frac{\alpha}{2}}$ is the fractional Laplacian given as the following non-local operator 
$$
\Delta^{\frac{\alpha}{2}}f(x):=\int_{\mR^d}\left(f(x+y)-f(x)-\1_{\{|y|\le 1\}}y\cdot\nabla f(x)\right)\frac{\dif y}{|y|^{d+\alpha}}.
$$

When $b$ is bounded and $u\to b(t,x,u)$ is Lipschitz uniformly in $t\in\mR_+$ and $x\in\mR^d$, the unique weak solution to \eqref{EQ:dDSDE:01} was constructed in \cite{WU2023SPA} as long as $\rho_0\in L^q$ with some $q>d/(\alpha-1)$. The aim of this paper is to approximate the solution to \eqref{EQ:dDSDE:01} and \eqref{FPE0} using the following moderately interacting 
$N$-particle system:
\begin{align}\label{MIP}
	\dif X_t^{N,i} &=b(t,X_t^{N,i},(\phi_N*\mu^N_t)(X_t^{{N,i}}))\dif t + \dif L^{\nalpha,i}_t,\quad i=1,\cdots,N,
\end{align}
where $\{L^{\alpha,i}\}_{i=1}^\infty$ is a family of i.i.d. standard $\mR^d$-valued $\alpha$-stable processes,
$$
\mu^N_t=\frac{1}{N}\sum_{i=1}^N \delta_{X^{N,i}_t}\quad \text{stands for the empirical measure,}
$$
and 
$$
\phi_N(x)=N^{\theta d}\phi(N^\theta x),\quad \text{for all $x\in\mR^d$}
$$
with some smooth compact supported probability density function $\phi$ and $\theta\in(0,\infty)$.

\subsection{Main results}
 Throughout the paper, we assume that 

 \noindent(\hypertarget{(H)}{$\boldsymbol{\rm H}$}) 
There are constants $ \kappa>0,$ and $\beta\in (0,1) $ such that for all \( \left(t, x,y, u,v\right) \in \mathbb{R}_{+} \times \mathbb{R}^{d} \times \mathbb{R}^{d}\times \mathbb{R}_{+}\times \mathbb{R}_{+} \),
\begin{equation*}
	\left|b\left(t, x, u\right)\right| \leqslant \kappa,\quad \text{and}\quad |b(t,x,u)-b(t,y,v)|\leq \kappa[|x-y|^{\beta}+|u-v|].
\end{equation*}
Moreover, \( \mu_{X_0}(\mathrm{d} x) = \rho_{0}(x) \mathrm{d} x \), where \( \rho_{0} \in L^{q}(\mathbb{R}^{d}) \) with some $q\in (\frac{d}{\alpha-1},\infty]$.

Under the condition (\hyperlink{(H)}{$\boldsymbol{\rm H}$}), a unique weak solution to dDSDE \eqref{EQ:dDSDE:01} on \( [0, T] \) was obtained for arbitrary time horizon \( T>0 \) in \cite{HAO2021JDE} for $\alpha=2$, and \cite{WU2023SPA} for $\alpha\in (1,2)$. Moreover, when $\rho_0\in \bC^\beta$ with $\beta>1-\alpha/2$, the pathwise uniqueness holds, and there is a unique strong solution (see \cite{WU2023SPA}). Here $\bC^\beta$ is the H\"older space, which will be introduced in Section \ref{sec:2}.

We present two main results in this paper. The first provides a quantitative convergence for the empirical measure, while the second concerns the convergence of the marginal single-particle.
\begin{theorem}\label{THM:CONVERGENCE:01}
  Let $\alpha\in(1,2]$, $T>0$, $\theta
  \in(0,\frac{1}{2d})$ and $\rho^N_t(x):=(\phi_N*\mu^N_t)(x)$. Assume that $($\hyperlink{(H)}{$\boldsymbol{\rm H}$}$)$ holds with {$\beta<\alpha-1-\frac{d}{q}$} and $\{X_0^{N,i}\}_{i=1}^N$ is a family of i.i.d. random variables with the common law $\mu_{X_0}$.
Then for any $m\in\mN$ and $\eps>0$, there exists a constant \( C=C(\alpha, T,\theta, \beta,q, \varepsilon, m)>0 \) such that for all \( N \geqslant 1 \) and $t\in(0,T]$,
\begin{equation}\label{EQ:CONVERGENCE:01}
\begin{split}
        \|\rho_{t}-\rho_{t}^{N}\|_{L^{m}\left(\Omega;L^{\infty}\right)} 
   & \le C{t^{-\frac{\beta+d/q}{\alpha}}}N^{-\theta\beta}+CN^{-1/2+\theta d+\eps}.
\end{split}
\end{equation}
Moreover, if $\rho_0\in\bC^\beta$, we can drop the condition $\beta<\alpha-1-\frac{d}{q}$ and have
\begin{equation}\label{EQ:CONVERGENCE:010}
\begin{split}
        \|\rho_{t}-\rho_{t}^{N}\|_{L^{m}\left(\Omega;L^{\infty}\right)} 
   & \lesssim N^{-\theta\beta}+N^{-1/2+\theta d+\eps}.
\end{split}
\end{equation}
\end{theorem}

\begin{theorem}\label{THM:MAINCONVERGENCE:01}
  Let $\alpha\in(1,2]$, $T>0$, and $\theta
  \in(0,\frac{1}{2d})$. Assume that $($\hyperlink{(H)}{$\boldsymbol{\rm H}$}$)$ holds and $\{X_0^{N,i}\}_{i=1}^N$ is a family of i.i.d. random variables with the common law $\mu_{X_0}$. 
  \begin{itemize}
      \item[(1)] {If $\beta<\alpha-1-\frac{d}{q}$}, then for any $\eps>0$, there exists a constant \( C=C(\alpha,T,\theta, \beta, q,m, \varepsilon)>0 \) such that for all \( N \geqslant 1 \),
\begin{equation}\label{EQ:CONVERGENCE:02}
\begin{split}
        \sup_{t\in[0,T]}\|\mP\circ(X^{N,1}_t)^{-1}-\mu_{t}\|_{\rm var} 
   & \le {{C}} N^{-\theta\beta}+CN^{-1/2+\theta d+\eps}.
\end{split}
\end{equation}
  \item[({2})] If $\beta>1-\alpha/2$ and $\rho_0\in \bC^\beta$,
then there are unique strong solutions $X$ and $(X^{N,1},...,X^{N,N})$ to dDSDE \eqref{EQ:dDSDE:01} driven by $L^{\alpha,1}$ and SDE \eqref{MIP} respectively. For any $m\in\mN$ and $\eps>0$, there exists a constant \( C=C(\alpha,T,\theta, \beta, m, \varepsilon)>0 \) such that for all \( N \geqslant 1 \),
\begin{equation} \label{EQ: THEOREM1.2-2}
\begin{split}
        \left\|\sup_{t\in[0,{T}]}|X^{N,1}_t-X_t|\right\|_{L^m(\Omega)} 
   & \le CN^{-\theta\beta}+CN^{-1/2+\theta d+\eps}.
\end{split}
\end{equation}
  \end{itemize}
\end{theorem}
\br
i) For the case (1), when $b=b(t,u)$ which is independent of $x$, and $u\to b(t,u)$ is Lipschitz, $($\hyperlink{(H)}{$\boldsymbol{\rm H}$}$)$ holds for all $\beta\in(0,1)$.

ii) For (2), the strong well-posedness of dDSDE \eqref{EQ:dDSDE:01} and SDE \eqref{MIP} are given in \cite{WU2023SPA}. 

iii) We can't address the case \(\theta \in [1/(2d),1/d)\), which is referred to as the “moderate” regime in \cite{Oelschlager1985} and represents an intermediate level of interaction. This regime will be the subject of future research.  

iv) In this paper, we establish quantitative approximation results for both the non-local nonlinear PDE \eqref{FPE0} and the density-dependent SDE \eqref{EQ:dDSDE:01} under the assumption that 
$b$ is H\"older continuous. To the best of our knowledge, for non-smooth $b$, there exist no prior quantitative results for either non-local quasilinear PDEs or dDSDEs driven by jump processes.

\er

\begin{example}
Consider the following non-local nonlinear FPE:
$$
\p_t\rho=\Delta^{\frac{\alpha}{2}}\rho+\div (b(\rho)\rho),
$$
where $b:\mR_+\to\mR^d$ satisfies $\sum_{i=1}^d|b'_i(r)|\leq\kappa$. 
Since the above equation can be written in the following transport form:
$$
\p_t\rho=\Delta^{\frac{\alpha}{2}}\rho+ (b(\rho)+b'(\rho)\rho)\cdot\nabla\rho,
$$
it is easy to see that by the maximum principle (see \cite[Theorem 6.1]{CHXZ} for example),
$$
\|\rho_t\|_\infty {\leq } \|\rho_0\|_\infty.
$$
Then our results can be applied rigorously by considering the truncated $b$ as 
$b_n(r)=b(r)\wedge n$, where $n>\|\rho_0\|_\infty$.
In particular, the above example covers the one dimensional fractional Burgers equation, i.e., $b(r)=r$.
In this case, if $\rho_0\in \bC^{\beta}$ with some $\beta\in(0,1)$, based on Theorem \ref{THM:CONVERGENCE:01}, for any smooth compact supported probability density function $\phi$ and $\theta\in(0,1/(2d))$,
\begin{align*}
    \sup_{x\in\mR^d}\left| \frac1{N^{1-d\theta}}\sum_{i=1}^N \phi(N^{\theta}(x-X^{N,i}_t))-\rho_t(x)\right|\lesssim N^{-\theta \beta}+N^{-1/2+\theta d+\eps},\ \text{ $\P$-a.s.}
\end{align*}
We believe that this is useful for numerical experiments.

\end{example}

\subsection{Related works and our contribution}

When $\alpha=2$, the following DDSDE has been extensively studied:
\begin{align}\label{DDSDE1}
    \dif X_t= (K*\mu_{X_t})(X_t)\dif t +\dif W_t,
\end{align}
where $W_t$ is a standard $d$-dimensional Brownian motion. The study of \eqref{DDSDE1} dates back to McKean's work \cite{McKean}, where a Lipschitz kernel $K$ was considered, and the pathwise convergence rate of order $N^{-1/2}$ for a single marginal particle was established. Later, under the assumption $K\in W^{-1,\infty}$ and additional conditions on $\div K$, a relative entropy estimate for the entire particle system was obtained in \cite{JW18}. More recently, in a general distributional dependent setting, the optimal convergence rate $N^{-1}$ in total variation distance as well as Wasserstein $W_2$ distance for a single marginal particle was proved in \cite{Lacker2023}, which includes \eqref{DDSDE1} with bounded $K$. It is extended for $K\in W^{-1,\infty}$ under further conditions on $\div K$ in \cite{Wa24}. 

In the special case where $K=\delta_0$ is the Dirac measure, the term $K*\mu_{X_t}(x)=\rho_t(x)$ reduces to the density $\rho_t(x)$, transforming DDSDE \eqref{DDSDE1} into a dDSDE. More generally, a class of McKean–Vlasov SDEs of Nemytskii type has been systematically studied in a series of works by Barbu and Röckner \cite{BR18, BR20, BR21, BR23, BR24} (see also their monograph \cite{BR}):
\begin{align}\label{MVSDENS}
    \dif X_t = b(t,X_t,\rho_t(X_t))\dif t+\sigma(t,X_t,\rho_t(X_t))\dif W_t,
\end{align}  
where $\sigma:\mR_+\times\mR^d\times \mR_+\to \mR^d \otimes \mR^d$ is a given measurable function. Under various assumptions on $b$ and $\sigma$, these works established existence and uniqueness results for the corresponding nonlinear Fokker–Planck equation and employed the superposition principle to construct weak solutions to \eqref{MVSDENS}.

For the special case $\sigma\equiv\mI_{d\times d}$, instead of using the superposition principle, a unique weak and strong solution was established in \cite{HAO2021JDE} via Euler approximations under the assumption that $b$ is bounded and the mapping $r\mapsto b(t,x,r)$ is Lipschitz continuous. This result was later extended in \cite{HRZ2024SIAMJMA}. Further well-posedness results were recently established in \cite{Le24}.

dDSDEs 
have gained importance in various applications, including physics (e.g., porous media equations \cite{BCR13}), biology (e.g., Fisher-KPP equations \cite{Flandoli2019JMAA}), and deep learning (e.g., diffusion models \cite{Zh24}). Since the singularity of the Dirac distribution, investigating propagation of chaos for \eqref{in:MIP} presents significant challenges.

To address this, we consider the following moderately interacting particle system:
\begin{align}\label{in:MIP} \dif X_t^{N,i}=b(t,X_t^{N,i},(\phi_N*\mu^N_t)(X_t^{N,i}))\dif t+\sigma(t,X_t^{N,i},(\phi_N*\mu^N_t)(X_t^{N,i}))\dif W_t^i,
\end{align}
where $\phi_N(x):=N^{\theta d}\phi(N^\theta x)$ for some probability density function $\phi$ and $\theta\in(0,\infty)$.

In \cite{Oelschlager1985}, Oelschl\"ager established a qualitative estimate for any \(\theta \in (0,1/d)\) under the Lipschitz condition on the mapping \((x,r) \mapsto b(t,x,r)\), along with certain additional conditions on \(\phi\). We also refer to \cite{MELEARD1987SPA} for related work. In \cite{JM98, HRZ2024SIAMJMA}, a quantitative estimate was obtained for \(\phi_N(x) = (\varepsilon_N)^d \phi (\varepsilon_N x)\), where \(\varepsilon_N \sim (\ln N)^{\theta}\) for some \(\theta \in (0,1/d)\). Moreover, in \cite{JM98}, a fluctuation estimate was established for $(\varepsilon_N)^2(\mu^N_t-\rho_t)$, also see  \cite{OePTRF} for a fluctuation result  in the moderate model setting.

However, it is worth noting that these studies do not consider 
$\alpha$-stable processes. When the driving noise in McKean–Vlasov SDEs is a jump process, such as an $\alpha$-stable L\'evy process, studying well-posedness and propagation of chaos becomes crucial in modeling physical phenomena such as the Boltzmann equation \cite{Tanaka, MM13} and the surface quasi-geostrophic model \cite[Section 6]{HRZAoP}. However, classical techniques like relative entropy methods, used in \cite{JW18, La21, Wa24}, are no longer applicable in this setting. This motivates further investigation.

For DDSDE \eqref{DDSDE1}, in the regime $\alpha\in(0,1)$, the well-posedness, Euler approximation, and propagation of chaos were established for H\"older continuous $K$ in \cite{HRW24}. When $K*\mu_{X_t}(x)$ is replaced by a general function $B(t,x,\mu_{X_t})$, density estimates and quantitative propagation of chaos were obtained in \cite{Ca24} under the conditions that $B$ is H\"older continuous with respect to both the spatial and measure variables. However, these results do not consider the density-dependent case.

For dDSDEs driven by $\alpha$-stable processes, the non-local conservation law was studied in \cite{SimonJDDE2018}, where the authors showed that the empirical process converges to a deterministic measure, which in turn solves a non-local PDE—but without a convergence rate. In \cite{HAOETAL2024ARXIV}, following the semigroup approach, the second-named author and his collaborators derived quantitative estimates for second-order moderately interacting particle systems with convolution case $K*\mu_{X_t}$ driven by $\alpha$-stable process, but the assumptions on $\theta$ did not allow reaching $\theta<1/(2d)$. Moreover, since $L^1$ is not a UMD space, their result can't cover the non-degenerate model \eqref{EQ:dDSDE:01}.


For dDSDEs \eqref{EQ:dDSDE:01} with $\alpha\in(1,2)$, the weak and strong well-posedness were established in \cite{WU2023SPA}, and a quantitative Euler approximation was derived in \cite{SONGandHAO2024ARXIV}.

{\subsection{Proof outline for main results} 
In this part, we provide an outline of the proof for main results, i.e., Theorem \ref{THM:CONVERGENCE:01} and Theorem \ref{THM:MAINCONVERGENCE:01}. The complete details can be found in Section \ref{SEC:EMPIRICAL} and Section \ref{sec:4} respectively. 

To establish Theorem \ref{THM:CONVERGENCE:01}, 
by It\^o's formula, we can observe that $\rho^N_t(x)=\frac{1}{N} \sum_{i=1}^N \phi_N (x-X_t^{N,i}) $ satisfies an stochastic partial differential equation (SPDE) (see \eqref{EQ:FPE:02} below). Meanwhile, our limit $\rho_t$ satisfies the nonlinear PDE \eqref{FPE0}. 
Then, it follows from the Duhamel's formula that the error function $\cU_t(x) = \rho_t(x) - \rho_t^N(x)$ satisfies
\begin{equation}\label{0329:00} 
	\begin{split}
		 \cU_t(x)&=-\int_{0}^{t} \nabla \cdot P_{t-s} [b(s,x,\rho_s(x)) \rho_s(x)-\langle  b(s,\cdot,\rho_s^N(\cdot)) \phi_N(x-\cdot), \mu_s^N(\cdot)  \rangle ] \dif s\\
		&\qquad -\int_{0}^{t} P_{t-s} \dif M_s^N(x)+P_t \cU_0^N(x),
	\end{split}
\end{equation}
where $P_t:=\exp(t\Delta^{\frac\alpha2})$ is the semigroup of the $\alpha$-stable process. 

 Next, we analyze each term in detail. For the first term, it can be decomposed into two components, one of which is controlled by $\|\cU_s\|_{L^\infty}$ itself, allowing us to apply Gronwall's inequality of Volterra-type. For the second component, the convergence rate with respect to $N$ is obtained using the H\"older regularity of $\rho_s$ and the drift $b$, where the H\"older regularity estimates for $\rho_s$, established in Lemma \ref{lem:25}, play a crucial role.

 For the stochastic integral term $\int_{0}^{t} P_{t-s} \dif M_s^N$, we apply the Burkholder-Davis-Gundy (BDG) inequality for Hilbert-valued martingales (cf. \cite{CR14, veraar2019pointwise}).

 Finally, for the initial value term $P_t \cU_0^N$, we decompose it into two parts: $P_t(\rho_0-\rho_0*\phi_N)$ and $P_t ((\rho_0-\rho_0^N)*\phi_N)$. The first term yields a convergence rate in $N$ via semigroup estimates in the Besov norm, while the second term is controlled using the independence of the initial data and the BDG inequality. 

As a result of Theorem \ref{THM:CONVERGENCE:01}, we can derive Theorem \ref{THM:MAINCONVERGENCE:01}. Specifically, note that the difference
$|b(s,x,\rho_s(x))-b(s,x, \rho_s^N(x))| $ is controlled by $ \| \rho_s-\rho_s^N \|_{L^\infty}$. This allows direct application of Theorem \ref{THM:CONVERGENCE:01}. For weak convergence, we use the It\^o-Tanaka trick. For pathwise convergence, we employ Zvonkin's transformation which is also used in \cite{HAOETAL2024ARXIV}.


}

\vspace{1em}

{\bf Structure of the paper}

\vspace{1em}

In Section \ref{sec:2}, we introduce H\"older and Besov spaces, which serve as the foundation for deriving heat kernel estimates. The most important result is presented in Lemma \ref{lem:25}. In Section \ref{SEC:EMPIRICAL}, following the approach used in \cite{HAOETAL2024ARXIV}, we give the proof of Theorem \ref{THM:CONVERGENCE:01}. Finally, in Section \ref{sec:4}, we prove our second main result Theorem \ref{THM:MAINCONVERGENCE:01}.  

\vspace{1em}

{\bf Notations}

\vspace{1em}

Throughout this paper, we use the symbol $C$ to denote constants, whose values may vary from one line to another. The notation $:=$ is used to signify a definition. We $A \lesssim B$ and $A \asymp B$ to indicate that there exists a constant $C \geq 1$ such that $A \leq C B$ and $C^{-1} B \leq A \leq C B$. We use the notation $a \sim b$ to express that $a$ and $b$ are of the same order, i.e., their ratio tends to $1$ asymptotically. Let $\mathbb{R}^d_{*}:= \mathbb{R}^d \setminus {\{0\}}$, $\R_{+}:=[0,\infty)$ and $B_r:=\{x\in \R^d:|x|\leq r\}$. The notation \( X \hookrightarrow Y \) is used to indicate that the space \( X \) can be embedded into the space \( Y \).
 The set of positive integers is denoted by $\N,$ and $\mathbb{N}_0$ is defined as $\N\cup\{0\}.$ Additionally,  we use $\mu_{X}$  to denote the distribution of the random variable $X$, and the notation $ X \stackrel{d}{=} Y$ is employed to signify that two random variables $X,Y$ are identically distributed. We use \( \|\mu-\nu\|_{\text {var }} \) to denote the total variation distance between probability measures \( \mu \) and \( \nu \).  The notation \( \mathbb{I}_{d \times d} \) denotes the \( d \)-dimensional identity matrix.
 Finally, $\mathcal{B}(E)$ represents the $\sigma$-algebra generated by the topology of the space $E$.

\section{Preliminaries}\label{sec:2}
In this section, we introduce some standard notations and heat kernel estimates that will be used later.

\subsection{H\"older and Besov spaces}
First, for $p\in[1,\infty)$, we use $\|\cdot\|_p$ to denote the usual $L^p$-norm.
For \( \beta \in(0,1] \) and \( f: \R^d \rightarrow \R \) we define the \( \beta \)-Hölder seminorm of \( f \) by
\begin{equation*}
	[f]_{\bC^{\beta}(\R^d)}:=\sup _{\substack{x, y \in \R^d \\ x \neq y}} \frac{|f(x)-f(y)|}{|x-y|^{\beta}}.
\end{equation*}
For \( \beta \in(0, \infty) \) we then denote by \( \bC^{\beta}(\R^d) \) the space of all functions such that for all \( \ell \in\left(\mathbb{N}_0\right)^{d} \) multi-indices with \( |\ell|<\beta \), the derivative \( \partial^{\ell} f \) exists, and
\begin{equation*}
	\|f\|_{\bC^{\beta}(\R^d)}:=\sum_{|\ell|<\beta} \sup _{x \in \R^d}\left|\partial^{\ell} f(x)\right|+\sum_{\beta-1<|\ell|<\beta}\left[\partial^{\ell} f\right]_{\bC^{\beta-|\ell|}(\R^d)}<\infty .
\end{equation*}
Hence the \( \bC^{\beta} \)-norm is stronger than the sup norm for any \( \beta>0 \). In the paper, we further define \( L^\infty(\mR^d):=\bC^{0}(\R^d) \) to be the space of all bounded measurable functions \( f: \R^d \rightarrow \R \) such that
\begin{equation*}
	\|f\|_\infty:=\|f\|_{\bC^{0}(\R^d)}:=\sup _{x \in \R^d}|f(x)|<\infty ,
\end{equation*}
By convention, for any $p\in[1,\infty]$, the space of $ p $-integrable functions on $ \mathbb{R}^d $ is denoted by $ L^p(\mathbb{R}^d) $, and the corresponding norm is denoted by $ \|\cdot\|_p $. When there is no ambiguity, $ L^p(\mathbb{R}^d) $ is simply denoted by $ L^p $.  Note that 
 $\|\cdot\|_{\bC^{0}}$ is actually  $\|\cdot\|_{\infty}.$

Let $\sS(\mR^{d})$ be the Schwartz space of all rapidly decreasing functions on $\mR^{d}$, and let $\sS'(\mR^{d})$ denote the dual space of $\sS(\mR^{d})$, known as Schwartz generalized function (or tempered distribution) space.
For any $f\in \sS(\mR^{d})$, we define the Fourier transform $\hat f$ and inverse Fourier transform $\check f$ respectively by
$$\hat f(\xi):=\ff{1}{(2\pi)^{d/2}}\int_{\mR^d}\e^{-i\xi\cdot x}f(x)\dif x, \ \xi\in\mR^d,$$
$$\check f(x):=\ff{1}{(2\pi)^{d/2}}\int_{\mR^d}\e^{i\xi\cdot x}f(\xi)\dif\xi, \ x\in\mR^d.$$
For $f\in\sS'(\mR^d)$, we define $\hat{f}$ and $\check{f}$ by the classical duality. 

To introduce the Besov space, we first introduce dyadic partitions of unity.  Let $\phi_{-1}$ be a symmetric
nonnegative $C^{\infty}$-function on $\mathbb{R}^d$ with
$$
\phi_{-1}(\xi)=1\ \mathrm{for}\ \xi\in B_{1/2}\ \mathrm{and}\ \phi_{-1}(\xi)=0\ \mathrm{for}\ \xi\notin B_{2/3}.
$$
For   $j\geq 0$, we define
\begin{align}\label{Phj}
\phi_j(\xi):=\phi_{-1}(2^{- (j+1)}\xi)-\phi_{-1}(2^{-j}\xi).
\end{align}
By 
definition, one sees that for $j\geq 0$, $\phi_j(\xi)=\phi_0(2^{-j}\xi)$ and
$$
\mathrm{supp}\,\phi_j\subset B_{2^{j+2}/3}\setminus B_{2^{j-1}},\quad\sum^n_{j=-1}\phi_j(\xi)=\phi_{-1}(2^{-(n+1)}\xi)\to 1,\quad n\to\infty.
$$

\bd
For  $j\geq -1$, the Littlewood-Paley block operator $\cR_j$ is defined on $\sS'({ \mR^d})$ by
$$
\cR_jf(x):=(\phi_j\hat f)\check{\,\,}(x)=\check\phi_j* f(x),
$$
with the convention $\cR_j\equiv0$ for $j\leq-2$.
In particular, for $j\geq 0$,
\begin{align}\label{Def2}
\cR_jf(x)=2^{ jd}\int_{{ \mR^d}}\check\phi_0(2^{j}y) f(x-y)\dif y.
\end{align}
\ed

For $j\geq -1$, by definition it is easy to see that
\begin{align}\label{KJ2}
\cR_j=\cR_j\widetilde\cR_j,\ \mbox{ where }\ \widetilde\cR_j:=\cR_{j-1}+\cR_{j}+\cR_{j+1},
\end{align}
and $\cR_j$ is symmetric in the sense that
$$
\<g, \cR_j f\>=\< f,\cR_jg\>,\ \ f,g\in\sS'(\mR^d),
$$
where $\<\cdot,\cdot\>$ stands for the dual pair between $\sS'(\mR^d)$ and $\sS(\mR^d)$.

Now we recall the definition of Besov spaces (see \cite{BCD11} for more details).
\bd\label{Def25}
Let $p,q\in[1,\infty]$ and $s\in\mR$. The Besov space $\bB^{s}_{p,q}$ is defined by
$$
\bB^{s}_{p,q}:=\left\{f\in\sS'(\mR^d): \|f\|_{\bB^{s}_{p,q}}:=
\left(\sum_{j\ge -1} 2^{sjq}\|\cR_j f\|_p^q\right)^{1/q}<\infty\right\}.
$$
 We denote $\bB^s_p:=\bB^s_{p,\infty}$. 
\ed

 For a function $f: \mathbb{R}^{d} \rightarrow \mathbb{R}$ and $h\in\mR^d$, the $ 1^{\text{st}} $-order difference operator is defined by
\begin{equation*}
	{\delta_{h}^{(1)} f(x):}=f(x+h)-f(x), \quad \forall x, h \in \mathbb{R}^{d} .
\end{equation*}
For $n \in \mathbb{N} $, the $n^{\text {th}} $-order difference operator is defined recursively by
\begin{equation*}
	\delta_{h}^{(n)} f(x)={\delta_{h}^{(1)}} \circ \delta_{h}^{(n-1)} f(x).
\end{equation*}
\br
For $s>0$, an equivalent characterization of $\bB^s_{p,q}$ is given by (see \cite[P74, Theorem 2.36]{BCD11} or \cite[Theorem 2.7]{HZZZ21})
$$
\|f\|_{\bB^{s}_{p,q}}\asymp\left(\int_{|h|\leq 1}
\left(\frac{\|\delta^{([s]+1)}_hf\|_p}{|h|^{s}}\right)^q\frac{\dif h}{|h|^d}\right)^{1/q}
+\|f\|_p.
$$
In particular,  for any $s\in(0,1)$ and $p\in[1,\infty]$, there is a constant $C=C(s,d,p)>0$ such that
$$
\|f(\cdot+h)-f(\cdot)\|_p\leq C\|f\|_{\bB^{s}_{p,\infty}} (|h|^s\wedge 1),
$$
and for any $s_0\in\mR$,
\begin{align}\label{Ho1}
\|f(\cdot+h)-f(\cdot)\|_{\bB^{s_0}_{p,\infty}}\leq C\|f\|_{\bB^{s_0+s}_{p,\infty}} (|h|^s\wedge 1).
\end{align}
From this estimate, for $s>0$ one sees that $\bB^s_\infty$ coincides with the classical H\"older space as long as $s\notin\mN$. Moreover, for any $n\in\mN$,
\begin{align*}
    \|f\|_{\bB^n_{\infty}}\lesssim_{n,d} \|f\|_{\bC^n}.
\end{align*}
We refer to \cite[Remark 3.4]{WU2023SPA} for more details on case $n\in\mN$.
\er

Below we recall some well-known facts about Besov spaces and $L^p$ spaces (see \cite[Lemma 2.4]{HZ24}).
\bl\label{lemB1}
\begin{enumerate}
\item For any $p\in[1,\infty]$,  $s'>s$ and $q\in[1,\infty]$, it holds that
\begin{align}\label{AB2}
\bB^{s'}_{p,\infty}\hookrightarrow \bB^{s}_{p,1}\hookrightarrow \bB^{s}_{p,q}\hookrightarrow \bB^{s}_{p,\infty},\ \bB^{0}_{p,1}\hookrightarrow L^p\hookrightarrow\bB^{0}_{p,\infty}.
\end{align}
\item For $1\leq p_1\leq p\leq\infty$, $q\in[1,\infty]$ and $\alpha\leq\alpha_1-\tfrac{d}{p_1}+\tfrac{d}{p}$, it holds that 
\begin{align}\label{Sob1}
\|f\|_{\bB^{\alpha}_{p,q}}\lesssim\|f\|_{\bB^{\alpha_1}_{p_1,q}}.
\end{align}
\end{enumerate}
\el
For $N\geq 1$ and $\theta>0$, we recall that
$\phi_N(x):=N^{\theta d} \phi(N^{\theta}x)$. Then similar as \cite[Lemmas B.1, B.2]{HAOETAL2024ARXIV} we have the following scaling inequality
\begin{equation}\label{EQ:DILATION:03}
  \|\phi_N\|_{\bB^{\beta}_{p,q}}\lesssim_{\beta,p,q,d} N^{\theta(\beta+d-\frac{d}{p})},\ \beta>0,\ p,q\in[1,\infty].
\end{equation}
Moreover, it follows from \eqref{Ho1} that for any $\beta\in\mR$, $p\in[1,\infty]$ and $\gamma\in(0,1)$, 
\begin{align}\label{0129:01}
    \|f*\phi_N-f\|_{\bB^{\beta}_p}\le \int_{\mR^d}\|\delta_y^{(1)}f\|_{\bB^\beta_p}\phi_N(y)\dif y\lesssim \||\cdot|^\gamma \phi_N\|_1\|f\|_{\bB^{\beta+\gamma}_p}\lesssim N^{-\gamma \theta}\|f\|_{\bB^{\beta+\gamma}_p}.
\end{align}

\subsection{$\alpha$-stable process}
For any $d$-dimensional \levy~process \( L \), i.e., a càdlàg process on \( \mathbb{R}^d \), such that \( L_0 = 0 \) almost surely, and the increments of \( L \) are independent and identically distributed,
 the associated Poisson random measure $\cN(\d s,\d z)$ is defined by
\begin{equation*}
	\cN((0, t] \times \Gamma):=\sum_{s \in(0, t]} \mathbbm{1}_{\Gamma}\left(L_{s}-L_{s-}\right), \quad \forall \Gamma \in \mathcal{B}\left(\mathbb{R}^{d}_{*}\right), t>0,
\end{equation*}
and the Lévy measure (i.e., the Poisson random measure's expectation during the unit time) is given by
\begin{equation*}
	\nu(\Gamma):=\mathbb{E} \cN((0,1] \times \Gamma) .
\end{equation*}
Then, the compensated Poisson random measure is defined by
\begin{equation*}
	\tilde{\cN}(\mathrm{d} s, \mathrm{d} z):=\cN(\mathrm{d} s, \mathrm{d} z)-\nu(\mathrm{d} z) \mathrm{d} s,
\end{equation*}
the compensator is defined by $\hat{\cN}(\d s,\d z) :=\nu(\d z)\d s.$

For \( \alpha \in(0,2) \), a \levy~process \( L^{\alpha} \) is called a symmetric and rotationally invariant \( \alpha \)-stable process if the \levy~ measure has the form
\begin{equation*}
    \nu^{(\alpha)}(\mathrm{d} z)=c|z|^{-d-\alpha} \mathrm{d} z
\end{equation*}
with some specific constant \( c=c(d, \alpha)>0 \). In this paper, we only consider the symmetric and rotationally invariant \( \alpha \)-stable process. Without causing confusion, we simply call it the \( \alpha \)-stable process, and assume that \( \nu^{(\alpha)}(\mathrm{d} z)=|z|^{-d-\alpha} \mathrm{d} z \) here and after.
It is easy to see that for any $\gamma_1<\alpha<\gamma_2$, 
\begin{equation}\label{EQ:0116:10}
    \int_{\R^d_*} \left[|z|^{\gamma_1}\wedge |z|^{\gamma_2}\right] \nu^{(\alpha)}(\mathrm{d} z) <\infty.
\end{equation}
By \cite[Proposition 2.5-(xii), Proposition 28.1]{SATO2013}, \( L_{t}^{\alpha} \) admits a smooth density function \( q_{\alpha}(t, \cdot) \) given by Fourier's inverse transform
\begin{equation}\label{EQ:KERNEL:01}
    q_{\alpha}(t, x)=(2 \pi)^{-d / 2} \int_{\mathbb{R}^{d}} \e^{-i x \cdot \xi} \mathbb{E} \e^{i \xi \cdot L_{t}^{\alpha}} \mathrm{d} \xi, \quad \forall t>0
\end{equation}
and the partial derivatives of \( q_{\alpha}(t, \cdot) \) at any orders tend to 0 as \( |x| \rightarrow \infty \).
Moreover, since the \( \alpha \)-stable process \( L \) is a self-similar process, 
\begin{equation*}
	\left(\lambda^{-1 / \alpha} L_{\lambda t}^{\alpha
    }\right)_{t \geqslant 0} \stackrel{d}{=}\left(L_{t}^{\alpha}\right)_{t \geqslant 0}, \quad \forall \lambda>0,
\end{equation*}
it turns out  that
\begin{equation*}
	q_{\alpha}(t, x)=t^{-d / \alpha} q_{\alpha}\left(1, t^{-1 / \alpha} x\right) .
\end{equation*}
It is well know that 
for any $k\in\mN_0$ and $p\in[1,\infty]$,
\begin{align*}
    \|\nabla^k q_\alpha(t,\cdot)\|_p=t^{-\frac{k+d/p}{\alpha}}\|\nabla^k q_\alpha(1,\cdot)\|_p\lesssim t^{-\frac{k+d/p}{\alpha}}.
\end{align*}
Note that \( q_{\alpha}(t, x) \) is also the heat kernel (or the fundamental solution) of the fractional Laplacian \( \Delta^{\alpha / 2} \), i.e., 
\begin{equation*}
    \partial_{t} q_{\alpha}(t, x)=\Delta^{\alpha / 2} q_{\alpha}(t, x), \quad \lim _{t \downarrow 0} q_{\alpha}(t, x)=\delta_{0}(x),
\end{equation*}
where \( \delta_0(x)\) is the Dirac-delta function. We also have the following Chapman-Kolmogorov (C-K) equations:
\begin{equation}\label{EQ:CKEQ:01}
    \left(q_{\alpha}(t) * q_{\alpha}(s)\right)(x)=\int_{\mathbb{R}^{d}} q_{\alpha}(t, x-y) q_{\alpha}(s, y) \mathrm{d} y=q_{\alpha}(t+s, x), \quad t, s>0.
\end{equation}

\subsection{Heat kernel estimates}
In this part, we always assume that \( \alpha \in(1,2) \) and the assumption (\hyperlink{(H)}{$\boldsymbol{\rm H}$}) holds. It follows from \cite{WU2023SPA} that there is a unique weak solution to dDSDE \eqref{EQ:dDSDE:01}. Moreover, by It\^o's formula, the density of the time marginal law of the solution satisfies the following nonlinear Fokker-Planck equation in the distributional sense:	
\begin{equation}\label{0105:00}
	\partial_t \rho_t(x)=\Delta^{\frac{\alpha}{2}} \rho_t(x)-\div(b(t,x,\rho_t(x)) \rho_t(x)).
\end{equation}
Let $P_t$ be the transition semi-group of the  \levy~process $L^{\alpha}$ whose generator is $\Delta^{\frac{\alpha}{2}}.$ Moreover, the action of the semigroup \( P_t \) on a function \( f \) can be characterized by the convolution of \( f \) with the heat kernel \( q_{\alpha}(t, \cdot) \), which is defined by  \eqref{EQ:KERNEL:01},
\begin{equation*}
	P_tf = q_{\alpha}(t,\cdot)*f.
\end{equation*}
Based on the Duhamel's formula (for example, see \cite[Lemma 3.1]{CHZ20}), we have
\begin{align}\label{0105:02}
    \rho_t(x)=P_t\rho_0(x)+\int_0^t P_{t-s}\div[b(s,\cdot,\rho_s)  \rho_s](x)\dif s.
\end{align}
Furthermore, the following estimate for the semigroup $P_t$ is well-known (cf. \cite[Lemma 2.14]{HRZAoP}): for any $\beta_1,\beta_2\in\mR$, $1\le p_1\le p_2\le \infty$ and $t\in[0,T]$,
\begin{align}\label{0129:02}
   \1_{ \{ \beta_2-\beta_1+d/p_1-d/p_2\ne0 \} }\|P_t f\|_{\bB^{\beta_2}_{p_2,1}}+\|P_t f\|_{\bB^{\beta_2}_{p_2}}\lesssim_{\beta_1,\beta_2,p_1,p_2,T}t^{-\frac{(\beta_2-\beta_1+d/p_1-d/p_2)\vee0}{\alpha}}\|f\|_{\bB^{\beta_1}_{p_1}}. 
\end{align}
In particular,
\begin{align}\label{0105:03}
    \|P_t f\|_{p_{2}}\lesssim_{p_1,p_2} t^{\frac{d}{\alpha p_{2}}-\frac{d}{\alpha p_{1}}}\|f\|_{p_{1}},\quad 1\le p_{1}\le p_{2}\le\infty,
\end{align}
and for any $0\le \gamma_1\le \gamma_2$,
\begin{align}\label{0105:04}
 \|P_t f\|_{\bC^{\gamma_2}}\lesssim_{\gamma_1,\gamma_2} t^{\frac{\gamma_1-\gamma_2}{\alpha }}\|f\|_{\bC^{\gamma_1}},
\end{align}
where for $k\in\mN_0$, $\|f\|_{\bC^k}:=\sum_{i=0}^k\|{\nabla^i }f\|_\infty$.

Here we give the following heat kernel estimates. 
\begin{lemma}[Density estimate]\label{lem:25}
 	Let $\beta\in(0,1)$ and $q\in(\frac{d}{\alpha-1},\infty]$ be given in $($\hyperlink{(H)}{$\boldsymbol{\rm H}$}$)$. Then for any $p\in [q,\infty]$, there is a constant $C=C(d,\alpha,\kappa,p,q)>0$ such that
\begin{equation}\label{0105:01}
	\|\rho_t\|_{p}\le C t^{\frac{d}{\alpha p}-\frac{d}{\alpha q}}\|\rho_0\|_{q}.
\end{equation}
Moreover, for any $\gamma\in (0,\alpha-1+\beta]$, there is a constant $C=C(d,\alpha,\kappa,p,q, \|\rho_0\|_{q})>0$ such that
\begin{equation}\label{0103:01}
    \|\rho_t\|_{\bC^{\gamma}}\le Ct^{-\big(\frac{\gamma
    }{\alpha}+\frac{d}{\alpha q}\big)},
\end{equation}
\end{lemma}
\begin{proof}
Based on \eqref{0105:02}, by \eqref{0105:03} we have
\begin{align*}
   \|\rho_t\|_{p}&\le  \|P_t\rho_0\|_{p}+\int_0^t \|\nabla P_{t-s}(b(s,\cdot,\rho_s)  \rho_s)\|_p\dif s\\
   &\lesssim t^{\frac{d}{\alpha p}-\frac{d}{\alpha q}}\|\rho_0\|_{q}+\int_0^t(t-s)^{-\frac{1}{\alpha}} \|b(s,\cdot,\rho_s)  \rho_s\|_p\dif s\\
   &\lesssim t^{\frac{d}{\alpha p}-\frac{d}{\alpha q}}\|\rho_0\|_{q}+\int_0^t(t-s)^{-\frac{1}{\alpha}} \|\rho_s\|_p\dif s,
\end{align*}
which by Gronwall's inequality of Volterra-type ({see, e.g., \cite[Example 2.4]{ZHANG2010JFA}}) implies \eqref{0105:01}.

For showing \eqref{0103:01}, given $\beta_0\in(0,\alpha-1-\frac{d}{q})$, we first prove
\begin{equation}\label{0105:05}
    \|\rho_t\|_{\bC^{\beta_0}}\lesssim t^{-\big(\frac{\beta_0
    }{\alpha}+\frac{d}{\alpha q}\big)}\|\rho_0\|_{q}.
\end{equation}
To this end, for any $s>0$, we introduce $\rho_{s,t}(x):=\rho_{s+t}(x)$. Then $t\to\rho_{s,t}$ solves \eqref{0105:00} with $b(t)=b(s+t)$ and $\rho_{s,0}=\rho_s$. Thus based on \eqref{0105:02}, it follows from \eqref{0105:04} that
\begin{align*}
    \|\rho_{s,t}\|_{\bC^{\beta_0}}&\lesssim \|P_t\rho_{s,0}\|_{\bC^{\beta_0}}+\int_0^t \|\nabla P_{t-r}(b(s+r,\cdot,\rho_{s,r})  \rho_{s,r})\|_{\bC^{\beta_0}}\dif r\\
    &\lesssim t^{-\frac{\beta_0}{\alpha}}\|\rho_{s,0}\|_{\infty}+\int_0^t(t-r)^{-\frac{1+\beta_0}{\alpha}} \|b(s+r,\cdot,\rho_{s,r})  \rho_{s,r}\|_{\infty}\dif r\\
    &\lesssim t^{-\frac{\beta_0}{\alpha}}\|\rho_{s,0}\|_{\infty}+\int_0^t(t-r)^{-\frac{1+\beta_0}{\alpha}} \|\rho_{s,r}\|_{\infty}\dif r.
\end{align*}
Noting that \eqref{0105:01} implies that
\begin{align*}
    \|\rho_{s,r}\|_{\infty}=\|\rho_{s+r}\|_{\infty}\lesssim (s+r)^{-\frac{d}{\alpha q}}{\|\rho_0\|_q},
\end{align*}
one sees that for $s=t/2$
\begin{align*}
    \|\rho_{t}\|_{\bC^{\beta_0}}&=\|\rho_{t/2,t/2}\|_{\bC^{\beta_0}}
    \lesssim \left(t^{-\frac{\beta_0}{\alpha}-\frac{d}{\alpha q}}+\int_0^{t/2}(t/2-r)^{-\frac{1+\beta_0}{\alpha}}(r+t/2)^{-\frac{d}{\alpha q}} \dif r\right)\|\rho_{0}\|_{q}\\
    &\lesssim\left(t^{-\frac{\beta_0}{\alpha}-\frac{d}{\alpha q}}+(t/2)^{-\frac{d}{\alpha q}}\int_0^{t/2}(t/2-r)^{-\frac{1+\beta_0}{\alpha}} \dif r\right)\|\rho_{0}\|_{q}\lesssim t^{-\frac{\beta_0}{\alpha}-\frac{d}{\alpha q}}\|\rho_{0}\|_{q},
\end{align*}
which is \eqref{0105:05}. Now we use induction to show that for any $k\in\mN$ and $\beta_k:=(k\beta_0)\wedge \beta$
\begin{equation}\label{0105:06}
    \|\rho_t\|_{\bC^{\beta_k}}\lesssim_{k,q} t^{-\big(\frac{\beta_k
    }{\alpha}+\frac{d}{\alpha q}\big)}(1+\|\rho_0\|_{q})^k.
\end{equation}
The method is similar as that showing \eqref{0105:05}: we assume that \eqref{0105:06} holds for $k$, consider $\rho_{s,t}$ again, and by \eqref{0129:02}, we  have 
\begin{align*}
    \|\rho_{s,t}\|_{\bC^{\beta_{k+1}}}&\lesssim
    \|P_t\rho_{s,0}\|_{\bC^{\beta_{k+1}}}+\int_0^t \|\nabla P_{t-{r}}(b(s+r,\cdot,\rho_{s,r})  \rho_{s,r})\|_{\bC^{\beta_{k+1}}}\dif r\\
   & \lesssim t^{-\frac{\beta_{k+1}}{\alpha}}\|\rho_{s,0}\|_{\infty}+\int_0^t(t-r)^{-\frac{1+\beta_0}{\alpha}} \|b(s+r,\cdot,\rho_{s,r})  \rho_{s,r}\|_{\bC^{\beta_k}}\dif r\\
   & \lesssim t^{-\frac{\beta_{k+1}}{\alpha}}\|\rho_{s,0}\|_{\infty}\\
&\quad+\int_0^t(t-{r})^{-\frac{1+\beta_0}{\alpha}}\left[\|b(s+r,\cdot,\rho_{s,r})  \|_{\bC^{\beta_k}}\|\rho_{s,r}\|_{\infty}+\|b(s+r,\cdot,\rho_{s,r})  \|_{\infty}\|\rho_{s,r}\|_{\bC^{\beta_k}}\right]\dif r,
\end{align*}
where we used the fact $\|fg\|_{\bC^\gamma}\lesssim \|f\|_{\bC^{\gamma}}\|g\|_{\infty}+ \|g\|_{\bC^{\gamma}}\|f\|_{\infty}$ for any $\gamma>0$. Since $\beta_k\le \beta$, based on the assumption \eqref{0105:06}, and by \eqref{0105:01}, we have for $s=t/2$,
\begin{align*}
    \|\rho_{t}\|_{\bC^{\beta_{k+1}}}
    &\lesssim t^{-\frac{\beta_{k+1}}{\alpha}-\frac{d}{\alpha q}}\|\rho_{0}\|_{q}+\int_0^{t/2}(t/2-r)^{-\frac{1+\beta_0}{\alpha}}\left[ \left(\|\rho_{t/2+r}  \|_{\bC^{\beta_k}} +1 \right)\|\rho_{t/2+r}\|_{\infty}+\|\rho_{t/2+r}\|_{\bC^{\beta_k}}\right]\dif r\\
   & \lesssim t^{-\frac{\beta_{k+1}}{\alpha}-\frac{d}{\alpha q}}\|\rho_{0}\|_{q}+\int_0^{t/2}(t/2-r)^{-\frac{1+\beta_0}{\alpha}}(t/2+r)^{-\frac{\beta_k+2d/q}{\alpha}}\dif r(1+\|\rho_{0}\|_{q})^{k+1}\\
    &\lesssim t^{-\frac{\beta_{k+1}}{\alpha}-\frac{d}{\alpha q}}\|\rho_{0}\|_{q}+(t/2)^{-\frac{\beta_k+2d/q}{\alpha}}\int_0^{t/2}(t/2-r)^{-\frac{1+\beta_0}{\alpha}}\dif r(1+\|\rho_{0}\|_{q})^{k+1}\\
     &\lesssim t^{-\frac{\beta_{k+1}}{\alpha}-\frac{d}{\alpha q}}\|\rho_{0}\|_{q}+(t/2)^{-\frac{\beta_k+2d/q}\alpha}(t/2)^{\frac{\alpha-1-\beta_0}{\alpha}}(1+\|\rho_{0}\|_{q})^{k+1}\\
    &\lesssim t^{-\frac{\beta_{k+1}}{\alpha}-\frac{d}{\alpha q}}(1+\|\rho_{0}\|_{q})^{k+1},
\end{align*}
provided that $ \alpha-1-\beta_0>d/q$,
which by induction gives \eqref{0105:06}. Hence, by taking $k$ large enough such that $(k\beta_0)\wedge \beta =\beta$, we have
\begin{equation}\label{0105:07}
    \|\rho_t\|_{\bC^{\beta}}\lesssim t^{-\big(\frac{\beta
    }{\alpha}+\frac{d}{\alpha q}\big)}.
\end{equation}
Based on \cite[Lemma 3.1]{CHZ20}, $v(t):=\int_0^t P_{t-s}f\dif s$ solves $(\p_t-\Delta^{\alpha/2} )v=f$, $v_0=0$. In view of, \cite[Lemma 3.6]{WU2023SPA} and \eqref{0105:02}, we have for any $s>0$,
\begin{align*}
    \|\rho_{s,t}\|_{\bC^{\beta+\alpha-1}}&\lesssim \|P_t\rho_{s,0}\|_{\bC^{\beta+\alpha-1}}+\sup_r\|b(s+r,\cdot,\rho_{s,r})  \rho_{s,r}\|_{\bC^{\beta}}\\
    &\lesssim t^{-\frac{\beta+\alpha-1}{\alpha}}\|\rho_{s,0}\|_{\infty}+\sup_r\|b(s+r,\cdot,\rho_{s,r})  \|_{\bC^{\beta}}\|\rho_{s,r}\|_\infty+\sup_r\|\rho_{s,r}\|_{\bC^\beta}\\
    &\lesssim t^{-\frac{\beta+\alpha-1}{\alpha}}\|\rho_{s,0}\|_{\infty}+\sup_r\|\rho_{s,r}  \|_{\bC^{\beta}}\|\rho_{s,r}\|_\infty+\sup_r\|\rho_{s,r}\|_{\bC^\beta}\\
    &\overset{\eqref{0105:01}}{\lesssim} t^{-\frac{\beta+\alpha-1}{\alpha}}s^{-\frac{d}{\alpha q}}\|\rho_{0}\|_{q}+\sup_r(s+r)^{-\frac{d}{\alpha q}}\|\rho_{s,r}  \|_{\bC^{\beta}}\\
    &\overset{\eqref{0105:07}}{\lesssim} t^{-\frac{\beta+\alpha-1}{\alpha}}s^{-\frac{d}{\alpha q}}+\sup_r(s+r)^{-\frac{\beta+2d/q}{\alpha}}\\
    &\lesssim t^{-\frac{\beta+\alpha-1}{\alpha}}s^{-\frac{d}{\alpha q}}+s^{-\frac{\beta+\alpha-1}{\alpha}-\frac{d}{\alpha q}},
\end{align*}
provided that $d/q<\alpha-1$. 
By choosing $s=t/2$ and invoking the interpolation property, we conclude the proof.
\end{proof}

The following result is derived from the uniqueness argument in part (ii) of the proof of Theorem 1.2 in \cite[page 441]{WU2023SPA}.
\bl
Let $\beta\in(0,1)$ be given in $($\hyperlink{(H)}{$\boldsymbol{\rm H}$}$)$. Then for any $\rho_0\in \bC^\beta$, there is a constant $C=C(d,\alpha,\kappa,\beta)>0$ such that
\begin{equation}\label{0220:00}
	\|\rho_t\|_{\bC^\beta}\le C \|\rho_0\|_{\bC^\beta}.
\end{equation}
\el

\section{Convergence of the empirical measure: Proof of Theorem \ref{THM:CONVERGENCE:01}}\label{SEC:EMPIRICAL}
{We recall
that $\rho_t(\cdot)$ solves the following nonlinear Fokker-Planck equation		
\begin{equation}\label{EQ:FPE:01}
	\partial_t \rho_t(x)=\Delta^{\frac{\alpha}{2}} \rho_t(x)-\div(b(t,x,\rho_t(x)) \cdot \rho_t(x)),\quad \lim_{t\downarrow 0}\rho_t(x)\d x =\mu_{X_0}(\d x) ~\text{weakly.}
\end{equation}
We also recall that $P_t$ is the transition semi-group of the  \levy~process $L^{\alpha}$ whose generator is $\Delta^{\frac{\alpha}{2}}$ 
and $\rho_t^N(x) :=(\phi_N * \mu_t^N)(x)$. 

 The main aim of this section is to show the main result Theorem \ref{THM:CONVERGENCE:01}. To do this, we rigorously analyze the difference between $\rho^{N}_{t}(x)$  and the true density $\rho_{t}(x)$.

Before presenting the proof, we provide a heuristic explanation for the convergence rate stated in Theorem \ref{THM:CONVERGENCE:01}, based on the decomposition \eqref{0329:00}. The right-hand side of \eqref{0329:00} consists of four distinct components:
\begin{enumerate}[label=(\alph*)]
 \item Initial data term $P_{t}\rho_{0}-P_{t}\rho_{0}^{N}$;
\item Reiteration term (settled for using Gronwall's inequality of Volterra-type);
\item Smallness term;
\item Martingale term (arising from \ito's formula).
\end{enumerate}
The reiteration term and smallness term originate from
\begin{align*}
    &\int_{0}^{t} \nabla \cdot P_{t-s} \left[ b(\rho_s(x)) \rho_s(x) - \left\langle b(\rho_s^N(\cdot)) \phi_N(x-\cdot), \mu_s^N(\cdot) \right\rangle \right] \dif s,
\end{align*}
where
\begin{align*}
   &\quad b(\rho_s(x)) \rho_s(x) - \left\langle b(\rho_s^N(\cdot)) \phi_N(x-\cdot), \mu_s^N(\cdot) \right\rangle \\
   &=b(\rho_s(x)) \rho_s(x) - \left\langle b(\rho_s(\cdot)) \phi_N(x-\cdot), \mu_s^N(\cdot) \right\rangle +  \left\langle \left[b(\rho_s(\cdot))-b(\rho_s^N(\cdot))\right] \phi_N(x-\cdot), \mu_s^N(\cdot) \right\rangle\\
   &= \left\{b(\rho_s(x))\rho_s^N(x)-\left\langle b(\rho_s(\cdot)) \phi_N(x-\cdot), \mu_s^N(\cdot) \right\rangle\right\}+b(\rho_s(x))[\rho_s(x)-\rho_s^N(x)] \\
   &\quad +  \left\langle \left[b(\rho_s(\cdot))-b(\rho_s^N(\cdot))\right] \phi_N(x-\cdot), \mu_s^N(\cdot) \right\rangle\\
   &=\text{smallness term}+\text{reiteration term 1}+\text{reiteration term 2}.
\end{align*}
The reiteration term 1 and 2 can be controlled by $\|\rho_s-\rho_s^N\|_{L^\infty}$, which can be absorbed using Gronwall's inequality. The smallness term, based on the density estimate in Lemma \ref{lem:25}, provides a convergence rate of $N^{-\beta\theta}$.

The initial data term can be further decomposed as
\begin{equation*}
    P_{t}\rho_{0}-P_{t}\rho_{0}^{N}=P_t(\rho_0 -\rho_0 * \phi_N)+P_t(\rho_0 * \phi_N - \mu_0^N * \phi_N).
\end{equation*}
The first term, $P_t(\rho_0 -\rho_0 * \phi_N)$, uses the regularity of the heat semigroup to yield a convergence rate of $N^{-\beta\theta}$ with a singularity in time of order $t^{-\frac{\beta+d/q}{\alpha}}$.

For the second component of the initial data term, $P_t(\rho_0 * \phi_N - \mu_0^N * \phi_N)$, as well as the martingale term (d), we can obtain a decay rate of $N^{-1/2 + \theta d}$. To ensure convergence of this term, we impose the condition $\theta < d/2$, which arise in our moderate interaction setting.

 We now proceed with the detailed proof.

} 

First of all, we recall the notation:
\begin{equation}\label{EQ:ERROR:01}
	\cU^N_t(x):=\rho_t(x)-\rho^N_t(x).
\end{equation}
By the definition of $\rho^N_t(\cdot)$,
\begin{equation}\label{EQ:MOLLIFIED:01}
\begin{split}
	\rho_t^N(x) =\frac{1}{N} \sum_{i=1}^{N} \phi_N(x-X_t^{N,i}) =\langle \phi_N(x-\cdot),\mu_t^N(\cdot) \rangle.
\end{split}
\end{equation}
Applying \ito's formula to $\phi_N(x-X_t^{N,i})$ for $1\leq i\leq N,$ we have
\begin{align*}
	\dif \phi_N(x-X_t^{N,i})&=-b(t,X_t^{N,i}, \phi_N \ast \mu_t^N(X_t^{N,i})) \cdot \nabla \phi_N (x-X_t^{N,i}) \dif t \\
	&\qquad+\Delta^{\frac{\alpha}{2}} \phi_N(x-X_t^{N,i}) \dif t + \d M_t^{N,i}(x).
\end{align*}
where
\begin{equation}\label{EQ:Martingale:00}
	M_t^{N,i}(x):=\begin{cases}
		\displaystyle \sqrt{2} \int_{0}^{t}\nabla\phi_N(x-X_s^{N,i})\d W_s^{i},& \alpha= 2,\\
		\displaystyle \int_{0}^{t} \int_{\mathbb{R}^d_{*}} [\phi_N(x-X_{s-}^{N,i}-z)-\phi_N(x-X_{s-}^{N,i})] \tilde{\cN}^i(\dif s, \dif z),& \alpha\in (1,2).
	\end{cases}
\end{equation} 
Therefore, $\rho_t^N$ solves the following SPDE
\begin{align*}
	\dif \rho_t^N(x) &=-\frac{1}{N} \sum_{i=1}^{N} b(t,X_t^{N,i}, \rho^N_t(X_t^{N,i})) \cdot \nabla \phi_N(x-X_t^{N,i}) \dif t+\Delta^{\frac{\alpha}{2}}  \rho_t^N(x) \dif t \\
	&\qquad +\frac{1}{N} \sum_{i=1}^{N} \d M_t^{N,i}(x).
\end{align*}
For simplicity, set $M_t^N(x) :=\frac{1}{N}\sum_{i=1}^{N}M_t^{N,i}(x) $. 
Then we have
\begin{equation}\label{EQ:FPE:02}
	\dif \rho_t^N(x)=\Delta^{\frac{\alpha}{2}} \rho^N_t(x) \dif t-\div_x \langle b(t,\cdot,\rho_t^N(\cdot)) \phi_N(x-\cdot) ,\mu_t^N(\cdot)  \rangle \dif t+\dif M_t^N(x).
\end{equation}
Based on the equations \eqref{EQ:FPE:01} and \eqref{EQ:FPE:02}, by the definition \eqref{EQ:ERROR:01}, one sees that
\begin{equation}\label{EQ:ERROR:02}
	\begin{split}
		\mathrm{d} \mathcal{U}_t^N(x) &=\Delta^{\frac{\alpha}{2}}\mathcal{U}_t^N(x) \mathrm{d} t -\div [b(t,\cdot,\rho_t(\cdot))\rho_t(\cdot)](x)\mathrm{d}t \\
		&\quad +\div_x\langle b(t,\cdot,\rho_t^{N}(\cdot))\phi_N(x-\cdot),\mu_t^N(\cdot) \rangle \mathrm{d} t-\mathrm{d} M_t^N(x).
	\end{split}
\end{equation}
By applying Duhamel's formula to \eqref{EQ:ERROR:02}, 
we have
\begin{equation}\label{EQ:ERROR:03}
	\begin{split}
		\cU^N_t(x)&=-\int_{0}^{t} \nabla \cdot P_{t-s} [b(s,x,\rho_s(x)) \rho_s(x)-\langle  b(s,\cdot,\rho_s^N(\cdot)) \phi_N(x-\cdot), \mu_s^N(\cdot)  \rangle ] \dif s\\
		&\quad -\int_{0}^{t} P_{t-s} \dif M_s^N(x)+P_t \cU_0^N(x) \\
		&=:-\int_{0}^{t} \nabla \cdot P_{t-s} H_s^N(x) \dif s- \int_{0}^{t} P_{t-s} \dif M_s^N(x)+P_t \cU_0^N(x),
	\end{split}
\end{equation}
where
\begin{align*}
	H^N_s(x)& =b(s,x,\rho_s(x)) \rho_s(x)-\langle  b(s,\cdot,\rho_s^N(\cdot)) \phi_N(x-\cdot), \mu_s^N(\cdot)  \rangle\\
	&= b(s,x,\rho_s(x))\rho_s(x)-b(s,x,\rho_s(x))\rho_s^N(x) +b(s,x,\rho_s(x))\rho_s^N(x)\\
	&\quad -\langle  b(s,\cdot,\rho_s^N(\cdot)) \phi_N(x-\cdot), \mu_s^N(\cdot)  \rangle\\
	&= b(s,x,\rho_s(x))\cU_s^N(x)+ \langle [b(s,x,\rho_s(x))- b(s,\cdot,\rho_s(\cdot))]\phi_N(x-\cdot),\mu_s^N(\cdot)\rangle \\
	&\quad +\langle [b(s,\cdot,\rho_s(\cdot))-b(s,\cdot
	,\rho_s^N(\cdot))]\phi_N(x-\cdot),\mu_s^N(\cdot)\rangle=:\sum_{i=1}^{3}H_s^{i,N}(x).
\end{align*}
It follows from the boundedness of $b$  that
\begin{equation*}
	|H_s^{1,N}(x)| =|b(s,x,\rho_s(x))\cU_s^N(x)|\lesssim |\cU_s^{N}(x)|.
\end{equation*}
As for $H_s^{2,N}(x),$ one sees that  
\begin{equation*}
	\begin{split}
		&\quad|b(s,x,\rho_s(x))-b(s,{y},\rho_s(y))| \\
        &= |b(s,x,\rho_s(x))-b(s,{x},\rho_s(y))+b(s,x,\rho_s(y))-b(s,{y},\rho_s(y))|\\
		& \lesssim |\rho_s(x)-\rho_s(y)| +|x-y|^{\beta}\\
		&\lesssim |x-y|^{\beta}[\|\rho_s\|_{{\bC}^{\beta}}+1].
	\end{split}
\end{equation*}
Noting that for any $x \in \mathbb{R}^{d}$,  $\operatorname{supp} \phi_{N} \subset\{x:|N^{\theta} x| \leqslant C\}= \{x:|x| \leqslant C N^{-\theta}\},$ which implies that
\begin{equation*}
	|x|^{\beta} \phi_{N}(x) \lesssim N^{-\theta\beta}\phi_{N}(x),
\end{equation*}
we have
\begin{equation*}
	\begin{split}
		H_s^{2,N}(x)
        &\lesssim\int_{\R^d} [\|\rho_s\|_{{\bC}^{\beta}}+1] \big(|x-y|^{\beta}\phi_N(x-y) \big)\mu_s^N(\d y)\\
		&\lesssim N^{-\theta\beta}(\|\rho_s\|_{{\bC}^{\beta}}+1)\int_{\R^d}\phi_N(x-y)\mu_s^N(\d y)\\
        &=  N^{-\theta\beta}(\|\rho_s\|_{{\bC}^{\beta}}+1) \rho_s^N(x).
	\end{split}
\end{equation*}
In the view of the boundedness of function $b$, we have
\begin{align*}
    H_s^{2,N}(x)\le \|b\|_\infty \<\phi_N(x-\cdot),\mu^N_s(\cdot)\>\lesssim \rho_s^N(x),
\end{align*}
which implies that
\begin{align*}
    H_s^{2,N}(x)&\lesssim [ 1\wedge (N^{-\theta\beta}(\|\rho_s\|_{{\bC}^{\beta}}+1))] \rho_s^N(x)\leq {[ 1\wedge (N^{-\theta\beta}(\|\rho_s\|_{{\bC}^{\beta}}+1))]} ({|\cU^N_s(x)|}+\rho_s(x))\\
    &\lesssim {|\cU^N_s(x)|}+N^{-\theta\beta}(\|\rho_s\|_{{\bC}^{\beta}}+1)\rho_s(x).
\end{align*}
For $H^{3,N}_s$, based on the Lipschitz continuity of $b(t,x,\cdot)$ and the boundedness of $b$, we also have
\begin{align*}
	|H_s^{3,N}(x)| &\lesssim [1\wedge\|\rho_s-\rho_s^N\|_{{\infty}}]|\langle \phi_N(x-\cdot),\mu_s^N(\cdot)\rangle| =[1\wedge\|\cU_s^N\|_{{\infty}}] |\rho_s^N(x)|\\
    &\lesssim[1\wedge\|\cU_s^N\|_{{\infty}}](|\cU_s^N(x)|+\rho_s(x))\lesssim |\cU_s^N(x)|+\|\cU_s^N\|_{{\infty}}|\rho_s(x)|.
\end{align*}
To sum up,
\begin{equation}\label{EQ:H:01}
	\begin{split}
		|H^N_s(x)| 
		& \lesssim |\cU_s^N(x)| +N^{-\theta\beta}(\|\rho_s\|_{{\bC}^{\beta}}+1) \rho_s(x) + |\rho_s(x)|\|\cU_s^N\|_{{\infty}} . 
	\end{split}
\end{equation}
Based on \eqref{EQ:H:01}, by the heat kernel estimate \eqref{0105:03}, we have
\begin{align*}
    \|\nabla P_{t-s}H^N_s\|_\infty
   & \lesssim (t-s)^{-\frac{1}{\alpha}}
    \|\cU^N_s\|_\infty+N^{-\theta\beta}(\|\rho_s\|_{{\bC}^{\beta}}+1)(t-s)^{-\frac{1+d/q}{\alpha}}\|\rho_s\|_{q}\\
    &\quad +\|\cU^N_s\|_\infty(t-s)^{-\frac{1+d/q}{\alpha}}\|\rho_s\|_{q},
\end{align*}
which by \eqref{EQ:ERROR:03} implies that
\begin{equation}\label{0120:00}
    \begin{split}
        \| \cU_t^N \|_{\infty}&\lesssim \| P_t \cU_0^N \|_{\infty}+ \int_{0}^{t}  \|\nabla P_{t-s} H^N_s \|_{\infty} \dif s+\left\|\int_{0}^{t} P_{t-s}\d M_s^N(\cdot)\right\|_\infty\\
   & \lesssim \| P_t \cU_0^N \|_{\infty}+ \int_{0}^{t} (t-s)^{-\frac{1}{\alpha}} [1+(t-s)^{-\frac{d}{q\alpha}}\|\rho_s\|_q] \| \cU^N_s \|_{\infty} \dif s\\
    &\quad+N^{-\theta \beta}\int_{0}^{t} (t-s)^{-\frac{1+d/q}{\alpha}}(\|\rho_s\|_{{\bC}^{\beta}}+1) \| \rho_s \|_{q} \dif s+\left\|\int_{0}^{t} P_{t-s}\d M_s^N(\cdot)\right\|_\infty.
    \end{split}
\end{equation}
Then it follows from the density estimates \eqref{0105:01} and \eqref{0103:01} that 
\begin{align}\label{EQ:U_t^N:04}
\begin{split}
       \| \cU_t^N \|_{\infty}& \lesssim \| P_t \cU_0^N \|_{\infty}+ \int_{0}^{t} (t-s)^{-\frac{1+d/q}{\alpha}} \| \cU^N_s \|_{\infty} \dif s\\
    &\quad+N^{-\theta \beta}\int_{0}^{t} (t-s)^{-\frac{1+d/q}{\alpha}} s^{-\frac{\beta+d/q}{\alpha}} \dif s+\left\|\int_{0}^{t} P_{t-s}\d M_s^N(\cdot)\right\|_\infty.
    \end{split}
\end{align}
By Gronwall's inequality of Volterra-type (see, e.g., \cite[Example 2.4]{ZHANG2010JFA}), we have 
\begin{align*} 
       \| \cU_t^N \|_{\infty}&\lesssim \int_{0}^{t} (t-s)^{-\frac{1+d/q}{\alpha}} \| P_s \cU_0^N \|_{\infty}\dif s+N^{-\theta \beta}\int_{0}^{t} (t-s)^{-\frac{1+d/q}{\alpha}}  s^{\frac{\alpha-1-\beta-2d/q}{\alpha}}\dif s  \no\\
    &\quad +\int_{0}^{t} (t-s)^{-\frac{1+d/q}{\alpha}} \left\|\int_{0}^{s} P_{s-r}\d M_r^N({\cdot})\right\|_\infty\dif s +\| P_t \cU_0^N \|_{\infty} \no\\
    &\quad +N^{-\theta \beta} \int_{0}^{t} (t-s)^{-\frac{1+d/q}{\alpha}} s^{-\frac{\beta+d/q}{\alpha}} \dif s + \left\|  \int_{0}^{t} P_{t-s} \dif M_s^N(\cdot)   \right\|_{\infty}.
\end{align*}
Noting that $\beta+\frac{d}q <\alpha-1<1+(\alpha-1)=\alpha$,  we have 
\begin{equation*}
    \int_{0}^{t} (t-s)^{-\frac{1+d/q}{\alpha}}  s^{\frac{\alpha-1-\beta-2d/q}{\alpha}}\dif s \leq t^{\frac{2\alpha-2-3d/q-\beta}{\alpha}}, 
\end{equation*}
and 
\begin{equation*}
    \int_{0}^{t} (t-s)^{-\frac{1+d/q}{\alpha}} s^{-\frac{\beta+d/q}{\alpha}} \dif s\leq t^{\frac{\alpha-1-\beta-{2d}/{q}}{\alpha}}.
\end{equation*}
It follows from $q>\frac{d}{\alpha-1}$ that 
\begin{equation}\label{EQ:0110:01}
    \begin{split}
        \| \cU_t^N \|_{\infty} &\lesssim \int_{0}^{t} (t-s)^{-\frac{1+d/q}{\alpha}}  \| P_s \cU_0^N \|_{\infty}\dif s+N^{-\theta \beta}t^{-\frac{d}{q\alpha}}+\left\|  \int_{0}^{t} P_{t-s} \dif M_s^N(\cdot)   \right\|_{\infty}\\
    &\quad+\int_{0}^{t} (t-s)^{-\frac{1+d/q}{\alpha}}\left\|\int_{0}^{s} P_{s-r}\d M_r^N({\cdot})\right\|_\infty\dif s+ \| P_t \cU_0^N \|_{\infty}.
    \end{split}
\end{equation}
Similarly, when $\rho_0\in\bC^\beta$, due to \eqref{0220:00}, we have $\|\rho_t\|_{\bC^\beta}\lesssim \|\rho_0\|_{\bC^\beta}\lesssim 1$, and by \eqref{0120:00} with $q=\infty$, we have
\begin{align*}
    \|\cU^N_t\|_\infty&\lesssim \|P_t\cU_0^N\|_\infty+\int_0^t (t-s)^{-\frac{1}{\alpha}}\|\cU^N_s\|_\infty\dif s+N^{-\theta\beta}\int_0^t (t-s)^{-\frac{1}{\alpha}}\dif s +\left\|\int_0^t P_{t-s}\dif M^N_s(\cdot)\right\|_{\infty}\\
    &\lesssim \|P_t\cU_0^N\|_\infty+\int_0^t (t-s)^{-\frac{1}{\alpha}}\|\cU^N_s\|_\infty\dif s+N^{-\theta\beta}+\left\|\int_0^t P_{t-s}\dif M^N_s(\cdot)\right\|_{\infty},
\end{align*}
which by Gronwall's inequality of Volterra-type (see, e.g., \cite[Example 2.4]{ZHANG2010JFA}) implies that
\begin{equation}\label{newEQ:0110:01}
\begin{aligned}
       \| \cU_t^N \|_{\infty}
    &\lesssim \int_{0}^{t} (t-s)^{-\frac{1}{\alpha}}  \| P_s \cU_0^N \|_{\infty}\dif s+N^{-\theta \beta}+\left\|  \int_{0}^{t} P_{t-s} \dif M_s^N(\cdot)   \right\|_{\infty}\\
    &\quad+\int_{0}^{t} (t-s)^{-\frac{1}{\alpha}}\left\|\int_{0}^{s} P_{s-r}\d M_r^N({\cdot})\right\|_\infty\dif s+ \| P_t \cU_0^N \|_{\infty}.
\end{aligned}
\end{equation}

\begin{lemma}
For any $\gamma\in(0,1),~m\geq 1$ and $\eps>0$, there exists a constant $C=C(\gamma,m,\eps)>0$ such that for any $t\in(0,T],$
    \begin{equation}\label{EQ:U_0^N:09}
        \left\|P_{t} \cU_{0}^{N}\right\|_{L^m(\Omega;L^{\infty})} \lesssim_C  t^{-\frac{\gamma}{\alpha}-\frac{d}{\alpha q}}N^{-\gamma\theta}\|\rho_0\|_q+N^{-\frac{1}{2}+\theta d+\eps}.
    \end{equation}
    Moreover, when $\rho_0\in\bC^{\gamma}$, it holds that
   \begin{equation}\label{newEQ:U_0^N:09}
        \sup_{t\in[0,T]}\left\|P_{t} \cU_{0}^{N}\right\|_{L^m(\Omega;L^{\infty})} \lesssim_C N^{-\gamma\theta}\|\rho_0\|_{\bC^\gamma}+N^{-\frac{1}{2}+\theta d+\eps}.
    \end{equation}
\end{lemma}

\begin{proof}
Due to \eqref{0129:02} and \eqref{AB2} 
we have for any $\gamma\in(0,1)$ and $\eps>0$
\begin{equation}\label{EQ:U_0:02}
	\begin{aligned}
		\left\|P_t \cU_{0}^{N}\right\|_{\infty}
		& \leqslant\left\|P_t \rho_{0}-P_t(\rho_{0} * \phi_{N})\right\|_{\infty}+\left\|P_t( \rho_{0} * \phi_{N}-\mu_{0}^{N} * \phi_{N})\right\|_{\infty} \\
		& \lesssim t^{-\frac{d}{\alpha q}} t^{-\frac{\gamma}{\alpha}} \left\|\rho_{0}-\rho_{0} * \phi_{N}\right\|_{\bB^{-\gamma}_q}+\left\|\left(q_{\alpha}(t,\cdot)  * \phi_{N}\right) *\left(\rho_{0}-\mu_{0}^{N}\right)\right\|_{\bB^{\eps}_{\infty,2}} \\
		& \lesssim {t}^{-\frac{d}{\alpha q}} t^{-\frac{\gamma}{\alpha}}N^{-\gamma\theta}\left\|\rho_{0}\right\|_{q}+\left\|\left(q_{\alpha}(t,\cdot)  * \phi_{N}\right) *\left(\rho_{0}-\mu_{0}^{N}\right)\right\|_{\bB_{2,2}^{\frac{d}{2}+\eps {/ \theta}}},
	\end{aligned}
\end{equation}
where we applied \eqref{0129:01} and \eqref{Sob1} in the last inequality.

Since $\phi_N,q_{\alpha}(t,\cdot)\in \sS(\mR^{d})$, we define
\begin{align*}
    Y_i(\omega):=(q_{\alpha}(t,\cdot)  * \phi_{N}) *\left(\rho_{0}-\delta_{X_0^{N,i}(\omega)}\right)
\end{align*}
as a random variable taking values in  $\bB^{d/2+\varepsilon}_{2,2}$. Noting that $\bB^{d/2+\varepsilon}_{2,2}$ is a Hilbert space, and applying the BDG inequality for Hilbert valued martingale (see, e.g., \cite[Theorem 16.1.1]{veraar2019pointwise}), we have
\begin{align*}
    \left\|\left(q_{\alpha}(t,\cdot)  * \phi_{N}\right) *\left(\rho_{0}-\mu_{0}^{N}\right)\right\|_{L^m(\Omega;\bB^{{d}/{2}+\eps}_{2,2})}^m&=\left\|\frac1N\sum_{i=1}^NY_i\right\|_{L^m(\Omega;\bB^{{d}/{2}+\eps}_{2,2})}^m \\
    &\lesssim N^{-m} \mE\left(\sum_{i=1}^N\|Y_i\|_{\bB^{\frac{d}{2}+\eps}_{2,2}}^2\right)^{\frac{m}{2}},
\end{align*}
where it follows from \eqref{EQ:DILATION:03} and Young convolution inequality related to $\bB_{p,q}^{s}$ (see, e.g, \cite[Lemma 2.6]{HRZAoP}) that
\begin{align*}
    \sup_i\|Y_i\|_{\bB^{\frac{d}{2}+\eps}_{2,2}}&= \sup_i \| (q_{\alpha}(t, \cdot) \ast \phi_N) \ast (\rho_0-\delta_{X_0^{N,i}(\omega)}) \|_{B^{\frac{d}{2}+\varepsilon}_{2,2}} \\
& \lesssim  \sup_i \| q_\alpha(t, \cdot) \ast \phi_N \|_{B^{\frac{d}{2}+\varepsilon}_{2,2}} \| \rho_0-\delta_{X_0^{N,i}(\omega)} \|_{B^0_{1, \infty}} \\
& \lesssim  \sup_i \| q_\alpha(t, \cdot) \|_{B_{1,\infty}^0}  \| \phi_N \|_{B^{\frac{d}{2} +\varepsilon}_{2,2}} \| \rho_0-\delta_{X_0^{N,i}(\omega)} \|_{B^0_{1, \infty}}\\
&\lesssim \|\phi_N\|_{\bB^{\frac{d}{2}+\eps}_{2,2}}\lesssim N^{\theta d+\eps \theta},
\end{align*}
which implies that
\begin{align*}
    \left\|\left(q_{\alpha}(t,\cdot)  * \phi_{N}\right) *\left(\rho_{0}-\mu_{0}^{N}\right)\right\|_{L^m(\Omega;\bB^{d/2+\varepsilon /\theta}_{2,2})}^m\lesssim N^{-m(\frac12-d\theta-\eps)},
\end{align*}
provided by taking $\eps=\eps/\theta$.
Thus, by \eqref{EQ:U_0:02}, we get \eqref{EQ:U_0^N:09}. Noting that when $\rho_0\in \bC^\gamma$, 
\begin{align*}
    \left\|P_t \rho_{0}-P_t(\rho_{0} * \phi_{N})\right\|_{\infty}\le\left\|\rho_{0}-\rho_{0} * \phi_{N}\right\|_{\infty}\lesssim N^{-\gamma\theta}\|\rho_0\|_{\bC^\gamma},
\end{align*}
we obtain \eqref{newEQ:U_0^N:09} and complete the proof.
\end{proof}

To proceed, note that the expression of 
\begin{equation}\label{EQ:M_t^N:01}
	\|\cM_t^N\|_{L^m(\Omega;L^{\infty})} =\left\{\E\left[\sup_{x\in \R^d}\Big|\int_{0}^{t}P_{t-s}\d M_s^N(x)\Big|^{m}\right]\right\}^{\frac{1}{m}}.
\end{equation}
varies with different values of index $\alpha.$ 
In the following our proof arguments are first focused on establishing general estimates on \( \cM_{t}^{N} \) successively in the Brownian case \( \alpha=2 \) (see Theorem \ref{THM:M_t^N:04} below) and next the pure-jump case \( \alpha \in(1,2) \) (see Theorem \ref{THM:M_t^N:05} below). 

We first consider the case  of $\alpha =2$. 
\begin{theorem}\label{THM:M_t^N:04}
	Let \( \alpha=2\). For any \( m > 1 \) and $\eps>0$, there is a constant \( C=C\left(\theta,  \beta, m,\eps\right)>0 \) such that, for any \( N \geqslant 1 \),
	\begin{equation*}
\left\|\cM_{t}^{N}\right\|_{L^{m}\left(\Omega ; {L^\infty}\right)} \lesssim_C N^{-\frac{1}{2}+\theta d+\eps}.
	\end{equation*}
\end{theorem}
\begin{proof}
In view of \eqref{AB2} and \eqref{Sob1}, one sees that for any $\eps>0$
\begin{align*}
			\E [\| \mathcal{M}_t^N \|_{\infty}^m ] \lesssim \E [\| \mathcal{M}_t^N \|_{\bB^{\eps}_{\infty,2}}^m ]\lesssim \E [\| \mathcal{M}_t^N \|_{\bB^{\frac{d}{2}+\eps}_{2,2}}^m ].		
		\end{align*}
Noting that $\bB^{d/2+\varepsilon}_{2,2}$ is a Hilbert space, based on
Hilbert-valued martingale’s BDG inequality (cf. \cite[Theorem 16.1.1]{veraar2019pointwise}) to the stopped martingale, $u\to\int_{0}^{u}P_{t-s}\d M_s^N$, $u\in[0,t]$, we have
	\begin{align*}
			\E [\| \mathcal{M}_t^N \|_{\infty}^m ] 		
            &\lesssim \frac{1}{N^m} \E\Big[\Big(\sum_{i=1}^N\int_{0}^{t}\|P_{t-s}\nabla \phi_N(\cdot-X_{s}^{N,i})\|_{{\bB^{\frac{d}{2}+\eps}_{2,2}}}^{2}\d s\Big)^{\frac{m}{2}}\Big]\\
            &\lesssim \frac{1}{N^{m}}\Big(\sum_{i=1}^{N}\int_{0}^{t}\|P_{t-s}\nabla \phi_N\|_{{\bB^{\frac{d}{2}+\eps}_{2,2}}}^2\d s\Big)^{\frac{m}{2}}.
		\end{align*}
       Then based on \eqref{0129:02}, it follows from \eqref{EQ:DILATION:03} that
        \begin{align*}
              \|P_{t}\nabla \phi_N\|_{{\bB^{\frac{d}{2}+\eps}_{2,2}}}&\lesssim \|\nabla \phi_N\|_{{\bB^{\frac{d}{2}+\eps}_{2,2}}}\wedge (t^{-\frac{d+1}{2}}\|\phi_N\|_{\bB^\eps_{1,\infty}})\lesssim N^{(d+1+\eps)\theta}\wedge (t^{-\frac{d+1}{2}}N^{\eps\theta})\\
              &\lesssim N^{(d+1+\eps)\theta}[1\wedge (N^{-\theta}t^{-\frac12})^{d+1}],
            \end{align*}
        which by the change of variable implies that
       \begin{align*}
			\int_{0}^{t}\|P_{t-s}\nabla \phi_N\|_{{\bB^{\frac{d}{2}+\eps}_{2,2}}}^2\d s&\lesssim N^{2(d+1+\eps)\theta}\int_0^t  1\wedge [N^{-\theta}(t-s)^{-\frac12}]^{ 2(d+1)}\dif s\\
            &\lesssim N^{2(d+1+\eps)\theta}N^{-2\theta}\int_0^\infty  1\wedge s^{  -(d+1)}\dif s\lesssim N^{(2d+2\eps)\theta}.
		\end{align*}
        Therefore, by taking $\eps=\eps/\theta$, we have
        \begin{align*}
			\E [\| \mathcal{M}_t^N \|_{\infty}^m ] 		
            &\lesssim \frac{1}{N^{m}}\Big(N^{1+2d\theta+2\eps}\Big)^{\frac{m}{2}}=N^{m(-\frac12+d\theta+\eps)}
		\end{align*}
        and finish the proof.
\end{proof}

Next, we consider the case when $\alpha \in (1,2)$. In this stage, 
\begin{equation*}
	\begin{split}
		\cM_t^N(x) := \int_{0}^{t}P_{t-s}\d M_s^N(x) =\frac{1}{N}\sum_{i=1}^{N}\int_{0}^{t}\int_{\R^d_{*}}\xi_{t}^{i}(s,z_i)(x)\tilde{\cN}^{i}(\d s,\d z_i),
	\end{split}
\end{equation*}
wherein 
\begin{equation}\label{EQ:XI::03}
	\begin{split}
		\xi_{t}^{i}(s,z)(x)&= P_{t-s} [\phi_N(x-X_{s-}^{N,i}-z)-\phi_N(x-X_{s-}^{N,i})]\\
		&= P_{t-s}\delta_{-z}^{(1)}\phi_N(\cdot-X_{s-}^{N,i})(x).
	\end{split}
\end{equation}

\begin{theorem}\label{THM:M_t^N:05}
	Let \( ~1<\alpha<2 \). Then for any \( \beta \geqslant 0,~m \geqslant 1 \) and \( \varepsilon>0 \), there exists \( C=C(T, \beta, m, \varepsilon)>0 \) such that for all \( N \geqslant 1 \),
	\begin{equation}\label{EQ:M_t^N:03}
		\|\cM_{t}^{N}\|_{L^{m}(\Omega;L^{\infty})} \lesssim_C N^{-{1}/{2}+\theta({d+\varepsilon)}}.
	\end{equation}
\end{theorem}
\begin{proof}
	To estimate  \eqref{EQ:M_t^N:01} when $\alpha\in(1,2)$, we take adapt the technique that developed in \cite{HAOETAL2024ARXIV}. So we need first to lift \( \cM^{N} \) on the product space \( \mathbb{R}_{*}^{N d}:= \mathbb{R}^{N d} \backslash\{0\} \) as follows, for
	\begin{equation*}
		\boldsymbol{L}_{t}^{N}:=(L_{t}^{\alpha, 1}, \cdots, L_{t}^{\alpha, N}),
	\end{equation*}
	the overall noise driving the particle system \eqref{MIP} and $ \boldsymbol{z}=(z_1,\cdots,z_N) \in \R^{Nd}$, let \( \boldsymbol{\cN}^{N}(\d s, \d \boldsymbol{z}) \) denote the jump measure of \( \boldsymbol{L}^{N} \) and \( \tilde{\boldsymbol\cN}^{N}(\d s,\d \boldsymbol{z}) \) the related compensated measure, respectively defined as: for all \( t \in(0, T] \) and \( \Gamma \in \mathscr{B}\left(\mathbb{R}_{*}^{N d}\right) \),
	\begin{equation*}
		\boldsymbol{\cN}^{N}((0, t], \Gamma):=\sum_{0<s \leqslant t} \mathbbm{1}_{\Gamma}\left(\Delta \boldsymbol{L}_{s}^{N}\right), \quad \tilde{\cN}^{N}(\mathrm{d} s, \mathrm{d} \boldsymbol{z}):=\cN^{N}(\mathrm{d} s, \mathrm{d} \boldsymbol{z})-\boldsymbol{\nu}(\mathrm{d} \boldsymbol{z}) \mathrm{d} s,
	\end{equation*}
	where \( \boldsymbol{\nu}(\mathrm{d} \boldsymbol{z}) \) is the \levy~ measure of \( \boldsymbol{L}^{N} \). Since the \( L^{\alpha, i} ,~1\leq i\leq N,\) are independent, their jumps \( \Delta L^{\alpha, i} \neq 0 \) never occur at the same time. This implies that the support of the \levy~ measure  $ \boldsymbol{\nu}(\d \boldsymbol{z}) $ in the product space $ \mathbb{R}_{*}^{N d} $ is restricted to the coordinate axes. It follows that \( \boldsymbol{\nu} \) and \( \boldsymbol{\cN}^{N} \) admit the following representations, respectively:
	\begin{equation}\label{EQ:NU:01}
		\begin{split}
			\boldsymbol{\nu}(\mathrm{d} \boldsymbol{z})=\sum_{i=1}^{N} \delta_{0}\left(\mathrm{d} z_{1}\right) \cdots \delta_{0}\left(\mathrm{d} z_{i-1}\right) \nu\left(\mathrm{d} z_{i}\right) \delta_{0}\left(\mathrm{d} z_{i+1}\right) \cdots \delta_{0}\left(\mathrm{d} z_{N}\right), \\
			\boldsymbol{\cN}^{N}(\d s, \mathrm{d} \boldsymbol{z})=\sum_{i=1}^{N} \delta_{0}\left(\mathrm{d} z_{1}\right) \cdots \delta_{0}\left(\mathrm{d} z_{i-1}\right) \cN^{i}\left(\d s, \mathrm{d} z_{i}\right) \delta_{0}\left(\mathrm{d} z_{i+1}\right) \cdots \delta_{0}\left(\mathrm{d} z_{N}\right),
		\end{split}
	\end{equation}
	for ${\delta_{0}}$ the Dirac measure in $0$. In particular, for any $1\leq {i}\leq N$, since ${\xi_{t}^{i}(s, 0)(x)=0}$ and  the measure ${\tilde{\cN}^{N}}$ only supports one jump at any given time, we have,
	\begin{equation*}
		\begin{split}
			&\quad 	\int_{\mathbb{R}_{*}^{N d}} \xi_{t}^{i}\left(s, z_{i}\right)(x) \tilde{\cN}^{N}(\mathrm{d} s, \mathrm{d} \boldsymbol{z}) \\
			&=\sum_{j=1, j \neq i}^{N} \int_{\mathbb{R}_{*}^{d}} \xi_{t}^{i}(s, 0)(x) \mathbbm{1}_{\left\{z_{j} \neq 0\right\}} \tilde{\cN}^{j}\left(\mathrm{d} s, \mathrm{d} z_{j}\right)+\int_{\mathbb{R}_{*}^{d}} \xi_{t}^{i}\left(s, z_{i}\right)(x) \tilde{\cN}^{i}\left(\mathrm{d} s, \mathrm{d} z_{i}\right) \\
			&=\int_{\mathbb{R}_{*}^{d}} \xi_{t}^{i}\left(s, z_{i}\right)(x) \tilde{\cN}^{i}\left(\mathrm{d} s, \mathrm{d} z_{i}\right).
		\end{split}
	\end{equation*}
	As such, if we next introduce the predictable process
	\begin{equation*}
		\boldsymbol{\xi}_{t}^{N}(s, \boldsymbol{z})(x):=\frac{1}{N} \sum_{i=1}^{N} \xi_{t}^{i}\left(s, z_{i}\right)(x), \quad  0 \leqslant s \leqslant t \leqslant T, x \in \mathbb{R}^{d},
	\end{equation*}
    where $\xi_t^i(s,z)$ is given by \eqref{EQ:XI::03},
	then ${\cM_{t}^{N}(z)}$ can be rewritten as
	\begin{equation*}
		\cM_{t}^{N}(x)=\int_{0}^{t} \int_{\mathbb{R}_{*}^{N d}} \boldsymbol{\xi}_{t}^{N}(s, \boldsymbol{z})(x) \tilde{\cN}^{N}(\mathrm{d} s, \mathrm{d} \boldsymbol{z}) .
	\end{equation*}
	Since $\bB^{\frac{d}{2}+\eps}_{2,2}$ is a Hilbert space with any $\eps>0$, applying  \cite[Theorem 1]{CR14} to the stopped martingale
	\begin{equation*}
		\cM_{u, t}^{N}(x)=\int_{0}^{u} \int_{\mathbb{R}_{*}^{N d}} \boldsymbol{\xi}_{t}^{N}(s, \boldsymbol{z})(x) \tilde{\cN}^{N}(\mathrm{d} s, \mathrm{d} \boldsymbol{z}), \quad  u \in[0, t].
	\end{equation*}
	We have, for any ${m \in \mathbb{N}}$ and $m\ge2$,	
	\begin{equation}\label{EQ:M_t^N:02}
		\begin{aligned}
			\mathbb{E}[\|\cM_{t}^{N}(\cdot)\|_{\infty}^{m}  & \lesssim \mathbb{E}[\sup _{0 \leqslant u \leqslant t}\|\cM_{u, t}^{N}(\cdot)\|_{\bB^{\frac{d}2+\eps}_{2,2}}^m] \\
			& \lesssim  \mathbb{E}\Big(\int_{0}^{t} \int_{\mathbb{R}_{*}^{N d}}\Big\|\boldsymbol{\xi}_{t}^{N}(s, \boldsymbol{z})(\cdot)\Big\|_{\bB^{\frac{d}2+\eps}_{2,2}}^{2} \boldsymbol{\nu}(\mathrm{d} \boldsymbol{z}) \mathrm{d} s\Big)^{m/2}\\
             &\quad +\mathbb{E}\int_{0}^{t} \int_{\mathbb{R}_{*}^{N d}}\Big\|\boldsymbol{\xi}_{t}^{N}(s, \boldsymbol{z})(\cdot)\Big\|_{\bB^{\frac{d}2+\eps}_{2,2}}^{m} \boldsymbol{\nu}(\mathrm{d} \boldsymbol{z}) \mathrm{d} s.
		\end{aligned}
	\end{equation}

	According to \eqref{EQ:NU:01}, for any $k\in\mN$,
	\begin{equation*}
		\int_{\mathbb{R}_{*}^{N d}}\left\|\boldsymbol{\xi}_{t}^{N}(s, \boldsymbol{z})(\cdot)\right\|_{\bB^{\frac{d}2+\eps}_{2,2}}^{{k}} \boldsymbol{\nu}(\mathrm{d} \boldsymbol{z})=\frac{1}{N^{{k}}} \sum_{i=1}^{N} \int_{\mathbb{R}_{*}^{d}}\left\|\xi_{t}^{i}\left(s,z_{i}\right)(\cdot)\right\|_{\bB^{\frac{d}2+\eps}_{2,2}}^{{k}} \nu^{(\alpha)}\left(\mathrm{d} z_{i}\right),
	\end{equation*}
	and then for $k\le m$,
	\begin{equation}\label{EQ:XI:01}
\begin{split}
    	&	\quad \mathbb{E}\Big[\Big(\int_{0}^{t} \int_{\mathbb{R}_{*}^{N d}}\|\boldsymbol{\xi}_{t}^{N}(s, \boldsymbol{z})(\cdot)\|_{\bB^{\frac{d}2+\eps}_{2,2}}^{k} \boldsymbol{\nu}(\mathrm{d} \boldsymbol{z}) \mathrm{d} s\Big)^{\frac{m}{k}}\Big]\\
            & =\frac{1}{N^m} \mathbb{E}\Big[\Big(\sum_{i=1}^{N} \int_{0}^{t} \int_{\mathbb{R}_{*}^{d}}\|\xi_{t}^{i}(s, z)(\cdot)\|_{\bB^{\frac{d}2+\eps}_{2,2}}^{{k}} \nu^{(\alpha)}(\mathrm{d} z) \mathrm{d} s\Big)^{\frac{m}{k}}\Big],
\end{split}
	\end{equation}
	where \eqref{Ho1} gives for $\eps\in(0,1-\alpha/2)$,
	\begin{equation}\label{EQ:XI:02}
		\begin{aligned}
			\left\|\xi_{t}^{i}(s, z)(\cdot)\right\|_{\bB^{\frac{d}2+\eps}_{2,2}} & =\left\|P_{t-s}\left(\delta_{-z}^{(1)} \phi_{N}\right)(X_{s-}^{N, i}-\cdot)\right\|_{\bB^{\frac{d}2+\eps}_{2,2}}=\left\|P_{t-s}\delta_{-z}^{(1)}  \phi_{N}\right\|_{\bB^{\frac{d}2+\eps}_{2,2}} \\
			& \lesssim \left(\| \phi_{N}\|_{\bB^{\frac{d}2+\eps}_{2,2}}\wedge \left[|z|^{\frac{\alpha+\eps}{2}}\|P_{t-s}  \phi_{N}\|_{\bB^{\frac{d+\alpha+3\eps}{2}}_{2,2}}\right]\right)\\
            & \lesssim \left(\| \phi_{N}\|_{\bB^{\frac{d}2+\eps}_{2,2}}\wedge \left[|z|^{\frac{\alpha+\eps}{2}}(t-s)^{-\frac{1}{2}}\|  \phi_{N}\|_{\bB^{\frac{d+3\eps}{2}}_{2}}\right]\right),
		\end{aligned}
	\end{equation}
    where we used \eqref{0129:02} in the last inequality.
	Hence, taking $\eps=2\eps/3$, noting the embedding relationship \eqref{AB2}, applying \eqref{EQ:DILATION:03} in \eqref{EQ:XI:02}, we have
    \begin{align*}
        \left\|\xi_{t}^{i}(s, z)(\cdot)\right\|_{\bB^{\frac{d}2+\eps}_{2,2}}\lesssim \|  \phi_{N}\|_{\bB^{\frac{d}{2}+\eps}_{2}}\left(1\wedge \left[|z|^{\frac{\alpha+\eps}{2}}(t-s)^{-\frac{1}{2}}\right]\right)\lesssim N^{\theta(d+\eps)}\left(1\wedge \left[|z|^{\frac{\alpha+\eps}{2}}(t-s)^{-\frac{1}{2}}\right]\right),
    \end{align*}
    which by \eqref{EQ:XI:01} and \eqref{EQ:M_t^N:02} yields
		\begin{align*}
			&\quad 	(\mathbb{E}[\|\cM_{t}^{N}\|_{\infty}^{{m}}])^{1 / m} \\
			& \lesssim  N^{-1+\theta(d+\eps)} \Big(\sum_{k=2,m} \Big(N\int_{0}^{t} \int_{\mathbb{R}_{*}^{d}} {\left(1\wedge \left[|z|^{\frac{\alpha+\eps}{2}}(t-s)^{-\frac{1}{2}}\right]\right)^k}\nu^{(\alpha)}(\mathrm{d} z) \mathrm{d} s\Big)^{\frac{m}{k}}\Big)^{\frac1{m}} \\
            &\lesssim  N^{-\frac12+\theta(d+\eps)} \sum_{k=2,m}\Big(\int_{0}^{t} \int_{\mathbb{R}_{*}^{d}} {\left(1\wedge \left[|z|^{\frac{\alpha+\eps}{2}}(t-s)^{-\frac{1}{2}}\right]\right)^k}\nu^{(\alpha)}(\mathrm{d} z) \mathrm{d} s\Big)^{\frac1{k}} \\
			&\lesssim  N^{-\frac12+\theta(d+\eps)} \sum_{k=2,m}\Big(\int_{0}^{t}(t-s)^{-\frac{\alpha}{\alpha+\eps}} \int_{\mathbb{R}_{*}^{d}} {(1\wedge |z'|^{\frac{\alpha+\eps}{2}})^k}\nu^{(\alpha)}(\mathrm{d} z') \mathrm{d} s\Big)^{\frac1{k}}\lesssim N^{-\frac12+\theta(d+\eps)},
		\end{align*}
    where we used a change of variable $z=(t-s)^{\frac1{\alpha+\eps}}z'$ in the last second inequality and \eqref{EQ:0116:10} in the last inequality. This completes the proof.
\end{proof}

Now we are ready to prove Theorem \ref{THM:CONVERGENCE:01}.
\begin{proof}[Proof of Theorem \ref{THM:CONVERGENCE:01}]
Notice
 \begin{equation*}
     \| \rho_t-\rho_t^N \|_{L^m(\Omega;L^\infty)} =\left\{  \E \left[ \| \mathcal{U}_t^N \|^m_{\infty} \right]  \right\}^{\frac{1}{m}},
 \end{equation*}
 where by \eqref{EQ:0110:01}, we have{
\begin{align*}
    \left\{\E \left[ \| \mathcal{U}_t^N \|^m_{\infty} \right] \right\}^{\frac{1}{m}} \lesssim & \left\{\E \left[\left(  \int_{0}^{t} (t-s)^{-\frac{1+d/q}{\alpha}} \| P_s \cU_0^N \|_{\infty}\dif s \right)^m \right]  \right\}^{\frac{1}{m}}+ N^{-\theta \beta}t^{-\frac{d}{q\alpha}}  \no \\
    & + \left\{\E \left[\left( \int_{0}^{t} (t-s)^{-\frac{1+d/q}{\alpha}} \left\|\int_{0}^{s} P_{s-r}\d M_r^N(\cdot)\right\|_\infty\dif s \right)^m \right] \right\}^{\frac{1}{m}} \no \\
    &  + \left\{\E \left[ \| P_t \cU_0^N \|_{\infty}^m \right]\right\}^{\frac{1}{m}}+  \left\{ \E \left[\left\|  \int_{0}^{t} P_{t-s} \dif M_s^N(\cdot)   \right\|_{\infty}^m \right] \right\}^{\frac{1}{m}}.
\end{align*}
}
Firstly, by taking $\gamma =\beta$ in \eqref{EQ:U_0^N:09} we have\\
\begin{align}\label{EQ:0114:07}
    \begin{split}
          \left\{\E \left[ \| P_t \cU_0^N \|_{\infty}^m \right]\right\}^{\frac{1}{m}} &=\| P_t \cU_0^N \|_{L^m(\Omega;L^\infty)} \\
   & \lesssim t^{-\frac{\beta}{\alpha}-\frac{d}{\alpha q}} N^{-\beta \theta} \|\rho_0 \|_q+ N^{-\frac{1}{2}+ \theta d+\varepsilon}
    \end{split}
\end{align}
Then, for $\alpha=2,$ by Theorem \ref{THM:M_t^N:04},
\begin{align*}
    \left\{ \E \left[\left\|  \int_{0}^{t} P_{t-s} \dif M_s^N(\cdot)   \right\|_{\infty}^m \right] \right\}^{\frac{1}{m}} &= \left\| \int_{0}^{t} P_{t-s} \dif M_s^N(\cdot)  \right\|_{L^m(\Omega;L^\infty)} \no \\
    & \lesssim N^{-\frac{1}{2}+\theta(d+\varepsilon)},
\end{align*}
and for $\alpha \in (1,2),$ due to \eqref{EQ:M_t^N:03},
\begin{align*}
    \left\{ \E \left[\left\|  \int_{0}^{t} P_{t-s} \dif M_s^N(\cdot)   \right\|_{\infty}^m \right] \right\}^{\frac{1}{m}} & \lesssim N^{-{1}/{2}+\theta({d+\varepsilon)}}.
\end{align*}
Moreover, by  Minkowski’s inequality  and \eqref{EQ:0114:07}, we have{
\begin{align*}
    \left\|\int_{0}^{t} (t-s)^{-\frac{1+d/q}{\alpha}} \| P_s \cU_0^N \|_{\infty}\dif s\right\|_{L^m(\Omega)} &\leq   \int_{0}^{t} (t-s)^{-\frac{1+d/q}{\alpha}}  \Big\| \| P_s \cU_0^N \|_{\infty} \Big\|_{{L^m(\Omega)}}\dif s\\
    &= \int_{0}^{t} (t-s)^{-\frac{1+d/q}{\alpha}}   \| P_s \cU_0^N \|_{L^m(\Omega;L^{\infty})} \dif s\\
    & \lesssim \int_{0}^{t}(t-s)^{-\frac{1+d/q}{\alpha}} \left(s^{-\frac{\beta}{\alpha}-\frac{d}{\alpha q}} N^{-\beta \theta}+ N^{-\frac{1}{2} + \theta d+ \varepsilon}\right) \d s\\
    &\lesssim t^{\frac{\alpha-1-\beta-2d/q}{\alpha}}N^{-\theta\beta} + t^{\frac{\alpha-d/q-1}{\alpha}} N^{-\frac{1}{2} + \theta d+ \varepsilon} \\
    &\lesssim t^{-\frac{d}{q\alpha}}N^{-\theta\beta} +N^{-\frac{1}{2} + \theta d+ \varepsilon}.
\end{align*}
}
To sum up, it follows that
\begin{equation*}
    \begin{split}
     \| \rho_t-\rho_t^N \|_{L^m(\Omega;L^\infty)} 
     &\lesssim t^{-\frac{\beta+d/q}{\alpha}}N^{-\theta\beta}+N^{-1/2+\theta d+\eps}.
    \end{split}
\end{equation*}
Similarly, \eqref{EQ:CONVERGENCE:010} is from \eqref{newEQ:0110:01}, \eqref{newEQ:U_0^N:09} and Theorem \ref{THM:M_t^N:04}-\ref{THM:M_t^N:05}.
The proof is complete.
\end{proof}

\section{Proof of Theorem \ref{THM:MAINCONVERGENCE:01}}\label{sec:4}
In this section, we use the convergence results of empirical measure given in Section \ref{SEC:EMPIRICAL} to prove Theorem \ref{THM:MAINCONVERGENCE:01}.

Before this, we outline the proofs of the weak and strong convergence results respectively, which proceed via distinct techniques tailored to each case. For the weak convergence, we employ the \ito-Tanaka trick. Specifically, we first consider the linearized PDE driven by $\rho_t$, where $\rho_t$ denotes the time marginal distributional density of the solution to the dDSDE. By applying It\^o's formula related to the solution of this linearized PDE, we derive a formulation quantifying the weak difference between the law of the particle system and the solution of the limiting equation (i.e., the solution to the dDSDE). Using regularity estimates for the linearized PDE's solution, we establish weak convergence in total variation. 

For the pathwise convergence, we employ the Zvonkin transformation, constructed via the solution to a backward PDE with well-established a priori estimates. By killing the singular drift coefficient $b$ and transforming it into a Lipschitz-continuous drift term, the Zvonkin transformation allows us to directly compare paths and derive pathwise convergence through a straightforward difference estimate.
\begin{proof}[Proof of Theorem \ref{THM:MAINCONVERGENCE:01}] (i) Fix $\varphi \in C_{b}^{\infty}\left(\mathbb{R}^{d}\right)$, the space of smooth and bounded functions defined on $\R^d$, and set $B_{t}^{T}(x):=B_{T-t}(x)=b(T-t,x,\rho_{T-t}(x))$. By \cite[Theorem 4.2-(i)]{HRZAoP}, there is a unique solution to the following PDE:
\begin{equation*}
    \partial_{t} u=\Delta^{\frac{\alpha}{2}} u+B^{T}_t \cdot \nabla u \quad \text { on }[0, T] \times \mathbb{R}^{d},  \quad {u}(0)=\varphi.
\end{equation*}
It follows from Duhamel's formula that
\begin{equation}\label{EQ:0111:02}
    u(t)=P_{t} \varphi+\int_{0}^{t} P_{t-s}[B_{s}^{T} \cdot \nabla u(s)] \mathrm{d} s,  \quad t \in[0, T].
\end{equation}
Noting
\begin{equation*}
    \|\nabla P_t f\|_{\infty}\lesssim t^{-\frac{1}{\alpha}}\|f\|_{\infty},
\end{equation*}
by \eqref{EQ:0111:02}, we have

\begin{align*}
    \left\|\nabla u(t)\right\|_{\infty}
    &\lesssim t^{-\frac{1}{\alpha}}\|\varphi\|_{\infty}+\int_{0}^{t}(t-s)^{-\frac{1}{\alpha}}\left\|B_{s}^{T}\right\|_{\infty}\left\|\nabla u(s)\right\|_{\infty} \mathrm{d} s \\
    &\lesssim t^{-\frac{1}{\alpha}}\|\varphi\|_{\infty}+\int_{0}^{t}(t-s)^{-\frac{1}{\alpha}}\left\|\nabla u(s)\right\|_{\infty} \mathrm{d} s,
\end{align*}
then by Gronwall's inequality of Volterra-type, we have
\begin{equation}\label{EQ:0111:04}
    \left\|\nabla u(t)\right\|_{{\infty}} \lesssim t^{-\frac{1}{\alpha}}\|\varphi\|_{{\infty}}.
\end{equation}
By applying the generalized version of \ito's formula to $t \mapsto u(T-t, Y_t)$ stated in \cite[Lemma 4.3]{HRZAoP} with $Y_t=X_t$ and $X^{N,1}_t$ respectively, we have
\begin{equation*} \mathbb{E}\left[\varphi\left(X_{T}\right)\right]=\mathbb{E}\left[u\left(0, X_{T}\right)\right]=\mathbb{E}\left[u\left(T, X_{0}\right)\right],
\end{equation*}
and
\begin{equation*}
    \begin{split}
        \mathbb{E}[\varphi(X_{T}^{N, 1})]
        &=\mathbb{E}[u(0, X_{T}^{N, 1})]\\
        &=\mathbb{E}\left[u(T,X_0 )\right]
+\mathbb{E}\Big[\int_{0}^{T}\left(b(s,X_s^{N,1},\rho_s^N(X_s^{N,1}))-b(s,X_s^{N,1},\rho_s(X_s^{N,1}))\right)\\
        &\qquad \qquad \qquad\qquad \qquad\cdot\nabla u(T-s,X_s^{N,1}) \mathrm{d} s\Big].
    \end{split}
\end{equation*}
Thus, by  \eqref{EQ:0111:04}, the Lipschitz continuity of $b(t,x,\cdot)$ and  \eqref{EQ:CONVERGENCE:01} for $m=1$,
\begin{align*}    |\mathbb{E}\left[\varphi\left(X_{T}\right)\right]-\mathbb{E}[\varphi(X_{T}^{N, 1})]| & \leqslant \mathbb{E}\left[\int_{0}^{T}|b(s,x,\rho_s^N(x))-b(s,x,\rho_s(x))|\left\|\nabla u(T-s)\right\|_{\infty} \mathrm{d} s\right] \\
& \lesssim\|\varphi\|_{\infty} \int_{0}^{T}(T-s)^{-\frac{1}{\alpha}}\left\|\rho_s^N-\rho_s\right\|_{L^1(\Omega;L^\infty)} \mathrm{d} s\\
&\lesssim \|\varphi\|_{\infty}\int_{0}^{T}(T-s)^{-\frac{1}{\alpha}}\Big(s^{-\frac{\beta+d/q}{\alpha}}N^{-\theta\beta}+N^{-1/2+\theta d+\eps}\Big)\d s\\
&\lesssim (N^{-\theta\beta}+N^{-1/2+\theta d+\eps})\|\varphi\|_{\infty},
\end{align*}
where $\beta+\frac{d}{q}<\alpha$.
Then \eqref{EQ:CONVERGENCE:02} follows from the following observation,
\begin{equation*}
    \begin{aligned}
\left\|\mathbb{P} \circ\left(X_{t}\right)^{-1}-\mathbb{P} \circ(X_{t}^{N,1})^{-1}\right\|_{\text {var}} =\sup _{\varphi \in C_{b}^{\infty}(\mathbb{R}^{d}) ;\|\varphi\|_{\infty}\leq 1}\left|\mathbb{E} \varphi\left(X_{t}\right)-\mathbb{E} \varphi(X_{t}^{N,1})\right| .
\end{aligned}
\end{equation*}

(ii) Without loss of generality, it suffices to prove the case for $m\geq 2$. The result for $m\le 2$ follows from the conclusion for $m\geq 2$ combined with Jensen's inequality. For any fixed $\lambda>0$, by \cite[Theorem 4.2-(i)]{HRZAoP}, there is a unique solution $u$ to the following (Zvonkin type) backward PDE:
\begin{equation*}
    \partial_{t} u+\Delta^{\frac{\alpha}{2}} u+B_t \cdot \nabla u=B_t,  \quad u(T)=0,
\end{equation*}
such that by \cite[Theorem 4.2-(iii)]{HRZAoP}, for $\lambda$ large enough, there exists $\varepsilon>0$ such that 
\begin{equation}\label{EQ:0111:10}
    \|\nabla u\|_{{L}_{T}^{\infty}}:=\|\nabla u\|_{L^{\infty}\left((0, T) ; {L}^{\infty}\right)} \leqslant \frac{1}{2},  \quad \|u\|_{L_T^{\infty}\bC^{{\alpha}/{2}+1+\varepsilon}}:=\|u\|_{L^{\infty}((0,T);\bC^{\alpha/2+1+\varepsilon})}<\infty.
\end{equation}
 Note that for each $t \in[0, T], x \mapsto \Phi_{t}(x):=x+u(t, x)$ forms a $C^{1}$-diffeomorphism on $\mathbb{R}^{d}$. By \ito's formula, we have for any $t \in[0, T]$,
\begin{equation*}
\begin{split}
        \Phi_{t}(X_{t}^{N, 1})&=\Phi_{0}\left(X_0\right)+\int_{0}^{t}\lambda u\left(s, X_{s}^{N, 1}\right) \mathrm{d} s+M_{t}^{N, 1}\\
        &\quad+\int_{0}^{t}\left[(b(s,\cdot,\rho_s^N(\cdot))-b(s,\cdot,\rho_s(\cdot))) \cdot \nabla \Phi_{s}\right]\left(X_{s}^{N, 1}\right) \mathrm{d} s
\end{split}
\end{equation*}
where
\begin{equation*}
    \begin{split}
        M_{t}^{N, 1}=\begin{cases}
\sqrt{2}\mathlarger{\int}_{0}^{t} \nabla \Phi_{s}\left(X_{s}^{N, 1}\right) \mathrm{d} W_{s}^{1}, &    \alpha=2, \\
\mathlarger\int_{0}^{t} \mathlarger \int_{\mathbb{R}^{d}_{*}} \delta_{z}^{(1)} \Phi_{s}\left(X_{s}^{N, 1}\right) \tilde{\cN}^{1}(\mathrm{d} s, \mathrm{d} z), &   \alpha \in(1,2).         
        \end{cases}
    \end{split}
\end{equation*}

Similarly, for $X_{t}^{1}$, the solution to \eqref{EQ:dDSDE:01} driven by $L^{\alpha, 1}$ and starting at the initial $X_{0}^{1}=X_{0}$, we have
\begin{equation*}  \Phi_{t}\left(X_{t}^{1}\right)=\Phi_{0}\left(X_0\right)+\int_{0}^{t}\lambda u\left(s, X_{s}^{1}\right) \mathrm{d} s+M_{t},
\end{equation*}
where
\begin{equation*}
    \begin{split}
        M_{t}=\begin{cases}
\sqrt{2} \mathlarger{\int}_{0}^{t} \nabla\Phi_{s}\left(X_{s}^{1}\right) \mathrm{d} W_{s}^{1},& \alpha=2, \\
\mathlarger{\int}_{0}^{t} \mathlarger{\int}_{\mathbb{R}^{d}_{*}} \delta_{z}^{(1)} \Phi_{s}\left(X_{s}^{1}\right) \tilde{\cN}^{1}(\mathrm{d} s, \mathrm{d} z),& \alpha \in(1,2).
        \end{cases}
    \end{split}
\end{equation*}
Thus, according to \eqref{EQ:0111:10},
\begin{equation}\label{EQ:0116:01}
    \begin{aligned}
\left|\Phi_{t}(X_{t}^{N, 1})-\Phi_{t}\left(X_{t}^{1}\right)\right| &  \lesssim \left(1+\lambda\|\nabla u\|_{{L}_{T}^{\infty}}\right) \int_{0}^{t}\left|X_{s}^{N, 1}-X_{s}^{1}\right| \mathrm{d} s+\left|M_{t}^{N, 1}-M_{t}\right| \\
& \quad +\left\|\nabla \Phi\right\|_{{L}_{T}^{\infty}} \int_{0}^{t}\left|[b(s,\cdot,\rho_s^N(\cdot))-b(s,\cdot,\rho_s(\cdot))]\left(X_{s}^{N, 1}\right)\right| \mathrm{d} s,
\end{aligned}
\end{equation}
where
\begin{equation*}
    M_{t}^{N, 1}-M_{t}=\begin{cases}
        \sqrt{2} \mathlarger{\int}_{0}^{t}\left[\nabla u\left(s, X_{s}^{N, 1}\right)-\nabla u\left(s, X_{s}^{1}\right)\right] \mathrm{d} W_{s}^{1},& \alpha=2, \\
\mathlarger{\int}_{0}^{t} \mathlarger{\int}_{\mathbb{R}^{d}_{*}}\left[\delta_{z}^{(1)} \Phi_{s}\left(X_{s}^{N, 1}\right)-\delta_{z}^{(1)} \Phi_{s}\left(X_{s}^{1}\right)\right] \tilde{\cN}^{1}(\mathrm{d} s, \mathrm{d} z),&  \alpha \in(1,2).
    \end{cases}
\end{equation*}
Observe in this latter case that:
\begin{equation}\label{EQ:0116:05}
   \begin{split}
              &\quad  \left|\delta_{z}^{(1)} \Phi_{s}\left(X_{s}^{N, 1}\right)-\delta_{z}^{(1)} \Phi_{s}\left(X_{s}^{1}\right)\right| \\
              & \lesssim \left|X_{s}^{N, 1}-X_{s}^{1}\right|\left[\|\nabla u(s, \cdot)\|_{{\infty}} \mathbbm{1}_{\{|z|>1\}}
        +\|\nabla u(s, \cdot)\|_{\bC^{(\alpha+\varepsilon)/2}}|z|^{(\alpha+\varepsilon)/2} \mathbbm{1}_{\{|z| \leqslant 1\}}\right] .
   \end{split}
\end{equation}
{
Then, in the case of $\alpha\in (1,2),$ it follows from BDG's inequality (see, e.g., \cite[Theorem 3.1]{Schilling2019SPA}) that
    \begin{align*}
\mathbb{E}[|M_{t}^{N, 1}-M_{t}|^{m}]&\lesssim
\E \left[\left(\mathlarger{\int}_{0}^{t} \mathlarger{\int}_{\mathbb{R}^{d}_{*}} \left[\left|\delta_{z}^{(1)} \Phi_{s}\left(X_{s}^{N, 1}\right)-\delta_{z}^{(1)} \Phi_{s}\left(X_{s}^{1}\right)\right|^{2} \right]\nu^{(\alpha)}(\mathrm{d} z) \mathrm{d} s\right)^{m/2}\right] & \\
&\qquad+\mathlarger{\int}_{0}^{t} \mathlarger{\int}_{\mathbb{R}^{d}_{*}} \E\left[\left|\delta_{z}^{(1)} \Phi_{s}\left(X_{s}^{N, 1}\right)-\delta_{z}^{(1)} \Phi_{s}\left(X_{s}^{1}\right)\right|^{m}\right] \nu^{(\alpha)}(\mathrm{d} z) \mathrm{d} s,
\end{align*}
which by \eqref{EQ:0116:05}, Jensen's inequality and \eqref{EQ:0111:10} implies that
\begin{align*}
\mathbb{E}[|M_{t}^{N, 1}-M_{t}|^{m}]&\lesssim 
\E \left[\left(\mathlarger{\int}_{0}^{t} \mathlarger{\int}_{\mathbb{R}^{d}_{*}} \left[\left(\left[|z|^{\frac{\alpha+\eps}{2}}\wedge 1\right]|X^{N,1}_s-X^1_s|\right)^2\|\nabla\Phi_s\|_{\bC^{(\alpha+\eps)/2}}^2 \right]\nu^{(\alpha)}(\mathrm{d} z) \mathrm{d} s\right)^{m/2}\right]\\
&\quad +  \mathlarger{\int}_{0}^{t}  \mathlarger{\int}_{\mathbb{R}^d_{*}} \E \left[ \left( |z|^{\frac{\alpha+\varepsilon}{2}} \wedge 1  \right)^m |X^{N,1}_s-X^1_s|^m \| \nabla \Phi_s \|^m_{\bC^{(\alpha+\varepsilon)/2}}  \right] \nu^{(\alpha)}(\d z) \d s
\\
&\lesssim 
  \E \mathlarger{\int}_{0}^{t}|X_s^{N,1}-X_s^1|^m\|\nabla u\|^{m}_{L_T^{\infty}\bC^{\alpha/2+\varepsilon}}\d s\mathlarger{\int}_{\R^d_{*}}(1\wedge |z|^{\alpha+\varepsilon})^{m/2} \nu^{(\alpha)}(\d z)\\
 & \quad +\mathlarger{\int}_{0}^{t} \E [|X^{N,1}_s -X_s^1 |^m] \d s\mathlarger{\int}_{\mathbb{R}^d_{*}} \left( |z|^{(\alpha+ \varepsilon)m/2} \wedge 1 \right) \nu^{(\alpha)} (\d z) \\
& \lesssim \int_{0}^{t} \mathbb{E}[|X_{s}^{N, 1}-X_{s}^{1}|^{m}] \mathrm{d} s,
\end{align*}
}
where in the last inequality, we use the fact \eqref{EQ:0116:10}, that is, when $\alpha \in(1,2)$ and $\gamma>\alpha,$
\begin{equation*}
    \int_{\mR^d_*}(1\wedge|z|^{\gamma} )\nu^{(\alpha)}(\mathrm{d} z)<\infty.
\end{equation*}
Moreover, by \eqref{EQ:0111:10} and \eqref{EQ:0116:01}, we have 
\begin{equation*}
\begin{split}
        \mathbb{E}\left[\sup _{s \in[0, t]}\left|X_{s}^{N, 1}-X_{s}^{1}\right|^{m}\right] & \lesssim \mathbb{E}\left[\sup _{s \in[0, t]}\left|\Phi_{s}\left(X_{s}^{N, 1}\right)-\Phi_{s}\left(X_{s}^{1}\right)\right|^{m}\right] \\
& \lesssim \int_{0}^{t} \mathbb{E}\left|X_{s}^{N, 1}-X_{s}^{1}\right|^{m} \mathrm{d} s\\
&\quad+\mathbb{E}\left[\left(\int_{0}^{T}|[b(s,\cdot,\rho_s^N(\cdot))-b(s,\cdot,\rho_s(\cdot))](X_s^{N,1})| \mathrm{d} s\right)^{m}\right],
\end{split}
\end{equation*}
which implies, by Gronwall's inequality, Minkowski' inequality and \eqref{EQ:CONVERGENCE:010}, that
\begin{align*}
     \nonumber \left\|\sup _{s \in[0, T]}\left|X_{s}^{N, 1}-X_{s}^{1}\right|\right\|_{L^m(\Omega)} & \lesssim \left\{\mathbb{E}\left[\left(\int_{0}^{T}|[b(s,\cdot,\rho_s^N(\cdot))-b(s,\cdot,\rho_s(\cdot))](X_s^{N,1})| \mathrm{d} s\right)^m\right]\right\}^{1/m} \\
     &=\left\|\int_{0}^{T}|[b(s,\cdot,\rho_s^N(\cdot))-b(s,\cdot,\rho_s(\cdot))](X_s^{N,1})| \mathrm{d} s\right\|_{L^m(\Omega)}\\
     &\lesssim  \int_{0}^{T}\| [b(s,\cdot,\rho_s^N(\cdot))-b(s,\cdot,\rho_s(\cdot))](X_s^{N,1}) \|_{L^m(\Omega)} \mathrm{d} s\\
& \lesssim \int_{0}^{T}\left\|\rho_{s}^{N}-\rho_{s}\right\|_{L^m(\Omega;L^\infty)} \mathrm{d} s\\
\nonumber 
&\lesssim \int_{0}^{T} \Big(N^{-\theta\beta}+N^{-1/2+\theta d+\eps}\Big) \d s\\
\nonumber 
&\lesssim  N^{-\theta\beta}+ N^{-1/2+\theta d+\eps}.
\end{align*}
This completes the proof for 
$\alpha\in (1,2)$; the case $\alpha=2$ follows similarly, thus concluding the proof.
\end{proof}

\subsection*{Acknowledgments}
We are deeply grateful to Prof. Rongchan Zhu for her valuable suggestions and for correcting some errors.

\bibliographystyle{plain}
\bibliography{reference_PoC_dDSDE_Levy}

\end{document}